\documentclass[12pt]{amsart}
\usepackage[utf8]{inputenc}
\usepackage[T1]{fontenc}
\usepackage[english]{babel}
\usepackage{amssymb,amstext,hyperref}
\usepackage{mathrsfs}
\usepackage{amsfonts,bm,bbm}
\usepackage{amsmath,verbatim}
\usepackage{xcolor}
\usepackage{enumitem}
\usepackage{tikz}

\usepackage{caption}
\usepackage{multicol}

\usepackage{amsthm}

\usepackage{geometry}
\definecolor{BlueGreen}{cmyk}{0.85,0,0.33,0}
\definecolor{dgreen}{rgb}{0.,0.6,0.}

\allowdisplaybreaks

\newtheorem{theorem}{Theorem}[section]
\newtheorem{corollary}[theorem]{Corollary}
\newtheorem{lemma}[theorem]{Lemma}
\newtheorem{proposition}[theorem]{Proposition}

\theoremstyle{definition}

\newtheorem{definition}[theorem]{Definition}

\newtheorem{remark}[theorem]{Remark}
\newtheorem{assumption}[theorem]{Assumption}
\newtheorem{example}[theorem]{Example}

\numberwithin{equation}{section}

\renewcommand{\epsilon}{\varepsilon}

\newcommand{\eps}{\varepsilon}

\newcommand{\calE}{\mathcal{E}}

\newcommand{\calD}{\mathcal{D}}
\newcommand{\calN}{\mathcal{N}}

\newcommand{\calK}{\mathcal{K}}

\newcommand{\calU}{\mathcal{U}}
\newcommand{\R}{\mathbb{R}}
\renewcommand{\P}{\mathbb{P}}
\newcommand{\E}{\mathbb{E}}

\newcommand{\bone}{\mathbbm{1}}

\newcommand{\LRa}{\Longleftrightarrow}
\newcommand{\mc}[1]{\mathcal{#1}}
\newcommand{\ld}{\lambda}

\newcommand{\prob}{\mathbb{P}}

\newcommand{\prt}{\partial}
\newcommand{\ds}{\displaystyle}

\DeclareMathOperator{\Leb}{Leb}
\DeclareMathOperator{\argmin}{argmin}

\begin{document}

\title{Billiards with Markovian reflection laws} 

\author{Clayton Barnes, Krzysztof Burdzy and Carl-Erik Gauthier}

\address{Department of Mathematics, Box 354350, University of Washington, Seattle, WA 98195}
\email{clayleroy2@gmail.com or clayton.barnes@unine.ch}
\email{burdzy@uw.edu}
\email{carlgaut@uw.edu or carlerik.gauthier@gmail.com}

\thanks{KB's research was supported in part by Simons Foundation Grant 506732. 
CEG's research was supported by the Swiss National Foundation for Research Grant P2NEP2\_171951. }


\pagestyle{headings}

 	\begin{abstract} 
We construct a class of reflection laws for billiard processes in the unit interval whose stationary distribution for the billiard position and its velocity is the product of the uniform distribution and the standard normal distribution. These billiard processes have Markovian reflection laws, meaning their velocity is constant between reflections but changes in a Markovian way at reflection times. 
 	\end{abstract}
 	\maketitle
 	
\section{Introduction}
\noindent
Consider a  billiard process $\{(X(t), L(t)), t\geq 0\}$ with values in $ [0, 1] \times \R$, where $X$ represents the billiard  position, reflecting at the endpoints 0 and 1, and $L$ represents the velocity of $X$. 
% (the reason for choosing $L$ as the symbol representing velocity will become clear later). 
Under the totally elastic collision assumption, i.e., when the kinetic energy is preserved, the long run  distribution of $X$ is uniform in $[0,1]$ and the speed, i.e. $|L|$, is constant. 

The Boltzmann-Gibbs distribution assigns a probability proportional to $\exp(-c \calE(x))$ to a state $x$ of a physical system, where $\calE(x)$ is the energy of the state $x$. This suggests that,  if the process $X$ does not move in a potential, i.e., the particle $X$ does not have potential energy, then the probability of a state $(x, \ell)$ of $(X,L)$, in the stationary regime, should be proportional to $\exp (-c \ell^2)$ because the kinetic energy is proportional to $\ell^2$. In other words,  position should be distributed uniformly in the interval $[0,1]$ and  velocity should be normally distributed in the stationary regime. For this  to be true, speed (i.e., the norm of velocity) must change at reflection times. We will present examples of Markovian reflection laws for billiard processes giving rise to the stationary density of the form $c_1\exp (-c_2 \ell^2)$. 

We will now state our main result and discuss related articles which inspired this research and provided some of the techniques used in this paper.

\subsection{Main result}

We will define our process on an interval $[0,T)$, possibly random, for some $0< T \leq \infty$, because we cannot assume from the outset that the process is well defined for all times $t>0$.

\begin{definition}\label{def:billiardProcess}
A process $\{(X(t), L(t)): t \in [0, T)\} $  with values in $ [0, 1] \times \R$ will be called a \emph{billiard process with Markovian reflections}
if and only if 
\begin{enumerate}[label = (\roman*)]
\item There exists an infinite sequence of random times $0=t_0<t_1<t_2 < \dots$ such that $\sup_i t_i  =T$.
\item $X(t) \in\{0,1\}$ if and only if $t=t_i$ for some $i$, a.s.
\item $L(t_0), L(t_1), \dots$, is a  Markov chain with non-zero values. The sign of $L(t_i)$ alternates, i.e., $L(t_i)L(t_{i+1}) < 0$ for every $i$. %if $L(t_i) > 0$ then $L(t_{i + 1}) < 0$, a.s. If $L(t_i) < 0$ then $L(t_{i + 1}) > 0$, a.s. 
\item  $L$ is  constant on $[t_i, t_{i + 1})$ for every $i$, a.s.
\item  $\ds X(t) = \int_0^tL(s)ds + X(0)$ for all $0\leq t<T$, a.s.
\end{enumerate}
\end{definition}

\medskip

A billiard process with Markovian reflections is a billiard whose 
velocity after reflection is random; it depends  on the incoming velocity and only on the incoming velocity.

Our notation for distributions and conditional distributions will be $\mc{L}(\,\cdot\,)$ and $\mc{L}(\,\cdot\,\mid\,\cdot\,)$.

 To construct a billiard process with Markovian reflections we need an initial condition $(X(0), L(0))$ and the Markov chain determining the laws of reflection.
This is sufficient  to construct a billiard process with Markovian reflections  on the time interval $[0, \sup_j t_j)$; see Section \ref{Sect3}.

Let  $\calU(0, 1)$ denote the uniform distribution on
$[0, 1]$ and let $\calN(0, 1)$ be the standard normal distribution on $\R$.
We will provide a large family of reflection laws $\mc{L}( L(t_{i + 1}) \mid L(t_i))$ for which $\calU(0, 1) \times \calN(0, 1)$
is the stationary distribution for the billiard process with Markovian reflections. The family will be indexed by an integer $N\geq 0$ and $\vec{\beta} = (\beta_0, \dots, \beta_N)\in (0,\infty)^{N+1}$.

\begin{definition}\label{def:vel}\
Consider an integer $N \geq 0$ and suppose that $\beta_i > 0$, $i = 0, \dots, N$ are reals such that $\sum_i\beta_i = 1$ and $\beta_i \geq \beta_{i+1} $ for all $i = 0, \dots, N$. Set $\ld_i = \beta_i/\sum_{j = 0}^i\beta_j$ for $i = 0, \dots, N$, and $\mu_i = \beta_{i + 1}/\sum_{j = 0}^i\beta_j, i = 0, \dots, N - 1$. 
Suppose that $\ell < 0$ and let
\begin{align}\label{s16.1}
p_N(\ell) &= \exp\left(-\frac{(1-\beta_N)\ell^2}{2\beta_N}\right),\\
p_k(\ell) &= 
\left(\prod_{m=k}^{N-1} \frac1{\mu_m}\right)
\sum_{j=k-1}^{N-1}
\exp\left(-\frac{\ell^2}{2\mu_j}\right)
\left(\prod_{\substack{m=k-1\\m\neq j}}^{N-1}\frac{1}{1/\mu_m- 1/\mu_j}\right), \quad k=1,\cdots, N-1, \label{s16.2}\\
p_0(\ell) &= 1 - \sum_{k=1}^N p_k(\ell). \label{s23.5}
\end{align}
Let
$Z(\ell)\in\{0,1, \dots, N\}$ be a random variable such that $\mathbb{P}(Z(\ell)=j)=p_j(\ell)$, $ j=0,\cdots, N$. Suppose that 
$E_0, E_1,\cdots , E_N$ are i.i.d.{} exponential, mean one, random variables
independent of $Z(\ell)$. Let $\mc{V}(\ell, N, \vec{\beta})$ be the distribution of
\begin{align}\label{s16.3}
 \sum_{j=0}^N \left(2\sum_{i=j}^N\lambda_i E_i\right)^{1/2} \bone_{Z(\ell)=j}.
\end{align}

We extend the definition of $\mc{V}(\ell, N, \vec{\beta})$ to $\ell>0$ by saying that $\mc{V}(\ell, N, \vec{\beta})$ is the distribution of $X$ if the distribution of $-X$ is $\mc{V}(-\ell, N, \vec{\beta})$.
\end{definition}

\begin{remark}\label{def:velold}
 Some factors in the product on the right hand side of \eqref{s16.2} are negative so it is not obvious that $p_k(\ell)$'s are non-negative. 
In fact, this is the case, as can be seen  from Proposition \ref{LevelDistrib} and Theorem \ref{ThinBoundaries} (i). The same results imply that $\sum_{k=0}^N p_k(\ell)=1$.
\end{remark}

\begin{example}\label{s17.1}
 Formulas \eqref{s16.1}-\eqref{s16.3} are complicated so we present three concrete examples.

(i)
In the case $N=0$ we necessarily have $\beta_0 =1$. If $\ell<0$, the distribution 
$\mc{V}(\ell, 0, (1))$ is the law of $\sqrt{2E_0}$, where $E_0$ is the exponential distribution with mean 1. For $\ell >0$, $\mc{V}(\ell, 0, (1))$ is the law of $-\sqrt{2E_0}$.

The distribution $\mc{V}(\ell, 0, (1))$  of $\sqrt{2E_0}$
is known as the Rayleigh distribution with parameter one.

(ii) Next consider the case $N=1$. Suppose that $0<\beta_1\leq 1/2$, $\ell<0$ and let
\begin{align}\label{s23.4}
p_0&= \exp\left(-\frac{(1-\beta_1)\ell^2}{2\beta_1}\right),\qquad
p_1= 1-\exp\left(-\frac{(1-\beta_1)\ell^2}{2\beta_1}\right).
\end{align}
Suppose that the following three random variables are independent: two mean-one exponentials $E_0$ and $E_1$, and $Z(\ell)$  such that $\mathbb{P}(Z(\ell)=j)= p_j(\ell)$, $j=0,1$. 
Although $\beta_0$ does not enter the following formula, we note that necessarily $\beta_0=1-\beta_1$. The distribution 
$\mc{V}(\ell, 1, (\beta_0,\beta_1))$ is the law of
\begin{align*}
\sqrt{2\left(E_0 + \beta_1 E_1\right)}\bone_{Z(\ell)=0}+ \sqrt{2\beta_1 E_1}\bone_{Z(\ell)=1}.
\end{align*}

(iii) Suppose that $N\geq 2$ and let $\beta_i = 1/(N+1)$ for $i=0,\dots, N$.
Then $\mu_i = \lambda_i = 1/(i+1)$. Elementary calculations show that
formulas \eqref{s16.1}-\eqref{s23.5} for $p_k(\ell)$
reduce to the binomial probabilities with parameters $N$ and $q:=\exp(-\ell^2/2)$, i.e., 
\begin{align*}
p_k(\ell) = \binom N k q^ k (1-q) ^{N-k}, \qquad k=0,1,\cdots, N.
\end{align*}
Formula \eqref{s16.3} becomes 
\begin{align*}%\label{s16.3}
 \sum_{j=0}^N \left(2\sum_{i=j}^N \frac{E_i}{i+1}\right)^{1/2} \bone_{Z(\ell)=j}.
\end{align*}

\end{example}

The next theorem is our main result.
Note that we may have different families of reflection laws for  reflections at 0 and 1. 

\begin{theorem}\label{mainTheorem}
Suppose that  $N^-, N^+\geq 0$ are integers, $\vec{\beta}^- \in (0,\infty)^{N^- + 1}$ and $\vec{\beta}^+ \in (0,\infty)^{N^+ + 1}$. 
\begin{enumerate}[label = (\roman*)]
\item There exists a billiard process $\{(X(t), L(t)) : t \in [0, \infty)\}$ with Markovian reflection laws
\begin{align*}
\mc{L}( L(t_{i + 1}) \mid L(t_i)=\ell)&=\mc{V}(\ell, N^-, \vec{\beta}^-),
\quad \text{  if  } \ell<0,\\
\mc{L}( L(t_{i + 1}) \mid L(t_i)=\ell)&=\mc{V}(\ell, N^+, \vec{\beta}^+),
\quad \text{  if  } \ell>0.
\end{align*}
\item $\calU(0, 1) \times \calN(0, 1)$ is the unique stationary distribution for $(X, L)$. 
\end{enumerate}
\end{theorem}

\begin{remark}
(i)  The reflection laws in Definition \ref{def:vel} correspond to the model considered in Theorem \ref{ThinBoundaries}. Two other results, Theorems \ref{Infinite_layers} and \ref{Cvge_Veloc}, implicitly contain two other families of reflection laws for which Theorem \ref{mainTheorem} holds. We do not give explicit formulas for these reflection laws  because they are more complicated than those in Definition \ref{def:vel}. The interested reader will have no problem with extracting definitions of those  reflection law families from the discussion  of the cases when $N_0(n)$ goes to infinity (in the ``noiseless case'') and the ``noisy case'' in Section \ref{Sect2}.

(ii)
In view of Example \ref{s17.1} (i) and Theorem \ref{mainTheorem} (ii), the stationary distribution of the process $(X, L)$ is $\calU(0, 1) \times \calN(0, 1)$ if  the speeds $|L(t_i)|$  are i.i.d. with the standard Rayleigh 
distribution. It is easy to see that this is the only example of a Markovian reflection law such that $L(t_{i+1})$ does not depend on $L(t_i)$ and the stationary distribution is $\calU(0, 1) \times \calN(0, 1)$. 
\end{remark}

\subsection{Proof strategy}
We will approximate a billiard with Markovian reflections by a sequence of processes $(X_n, L_n)$ with state spaces $\calD_n \times \R$, where $\calD_n = \{0, 1/n, 2/n, \dots, 1\}$. The processes $(X_n, L_n)$ will belong to a particular class introduced in \cite{BW1}, which we describe in Section \ref{Sect1}. These processes  have the stationary distribution  $\calU(\calD_n) \times \calN(0, 1)$, where $\calU(\calD_n)$ denotes the uniform distribution on $\calD_n$. Consequently, if $(X_n, L_n)$ converges to a process, classical limit results show that $\calU(0, 1) \times \calN(0, 1)$ is the stationary 
distribution for the limit. Constructing  a sequence of processes that converge to a billiard process with 
Markovian reflections relies, roughly speaking, on a finite system of equations involving  transition rates 
between states of $X_n$ together with a process $L_n$ which represents the  ``memory'' of $X_n$. Manipulation of these equations will give rise to the variety of reflection laws described above.

As we have already mentioned, we will approximate the unit interval with the discrete interval $\{0, 1/n, 2/n, \dots, 1\}$. We will reserve 
a tiny fraction of these points to serve the role of boundaries, namely the first $N_0(n)+1$ (resp. the last $N_1(n)+1$)
points will form the ``boundary'' at $0$ (resp.  at $1$). We think of these short discrete intervals as 
layers in which the random reflection  takes place. In the limit, the layers will collapse to the respective endpoints. Thus 
we take $N_j(n)/n\rightarrow 0$ as $n$ converges to infinity, for $j=0,1$. For fixed $n$, one can think of the points in $[0, N_0/n] \cup [(n - N_1)/n, 1]$ as holding a potential that  reverses the direction of the motion of the particle $X_n$ as it 
approaches either boundary. After this reversal it will leave the potential layer with a random ``velocity.'' These potential layers will disappear as $n$ 
approaches infinity. Because of this, the limiting process will have ballistic trajectories, but randomness for 
the reflecting velocity will be retained. The 
``velocity'' $L_n$ will not change outside of the potential layers in our model. 

To make the model tractable, we will consider only two dynamics
inside the boundary layers. In the first case, the particle $X_n$ will be able to jump  in only one direction, depending on the sign of $L_n$. In the second case, $X_n$ will be able to jump to both neighbors  but the boundary layers will be very thin, i.e., $N_0(n)= N_1(n)=1$.  

\begin{remark}
The formula for the reflected velocity \eqref{s16.3} is complicated and hard to comprehend intuitively. One may wonder whether a more accessible examples may arise by passing with $N$ to infinity and scaling $\vec{\beta}$  appropriately. This does not seem to be the case. The limit seems to be deterministic. In other words, the  limiting reflection would be totally elastic, resulting in the constant speed  for all times. The reason is that
 Theorem \ref{mainTheorem} is based on a ``noiseless'' approximation scheme where the particle can jump  in only one direction, depending on its current drift. For large $N$, the law of large numbers would generate deterministic reflections.

We expect that in the ``noisy'' case, when the particle can jump in both directions, there may exist interesting limiting distributions. However, we can effectively  analyze  the ``noisy'' case only for $N=1$.
\end{remark}

\subsection{Related results}
We have already indicated, at the beginning of the introduction, that our research is inspired by certain ideas from physics. On the mathematical side, this paper is related to models of Markov processes with ``memory''
presented in \cite{BBCH,BCG15,CEGMB,BKS12,BKS13}.
We will not review these models in detail because they are quite diverse. What they have in common is that, in every case, the stationary distribution has the product form---it is uniform (on an appropriate space) for the ``position'' component  of the process and it is Gaussian for the ``memory.'' The product form of the stationary distribution is far from obvious because the components, position and memory, are not independent; they are not even Markov on their own.
In view of the history of the model, we will interchangeably refer to the second component of $(X_n, L_n)$ as ``velocity'' or ``memory.'' 

The perspective of this paper is 
the reversal of the classical problem of finding the stationary distribution. We are looking for models that have the prescribed product-form stationary distribution.

Our specific model has the following roots.
In \cite{BBCH}, a reflected Brownian motion with drift  was analyzed.
The drift had memory---it accumulated proportionally to the vector-valued local time on the boundary. As a part of the analysis, the authors of \cite{BBCH} considered a sequence of Brownian motions not reflected on the boundary but repulsed by a sequence of smooth potentials converging to 0 inside the domain and to infinity on the boundary. The diffusion coefficient remained constant.
One may wonder what limiting processes could arise if we let the potentials converge in the manner described above and at the same time we let diffusivity go to 0 at an appropriate rate. It is clear that the limiting process must have ballistic trajectories inside the domain but its reflection law might be random. Our present article can be viewed as a simplified version of the problem, but one that tries to go into the heart of the matter.

At the technical level, we will use a discrete approximation, originally introduced in \cite{BW1}. So far, this type of approximation was used only for generating conjectures which were subsequently proved using other methods, as in \cite{BKS12} and \cite{BKS13}. Convergence of a discrete approximation of this type to a Markov process with memory was proved for the first time in \cite{B}.

Finally, we would like to 
point out that \cite{BW2} presented a process with sawtooth paths, just like our process $X$. In that case, the sawtooth process had a Gaussian stationary distribution. The speed was constant and the locations of direction reversals were random, whereas in our case, the locations of direction reversals are fixed but the speed is random. 

\subsection{Organization of the paper}

In Section \ref{Sect1} we introduce  approximating processes and state our assumptions. In Section \ref{Sect2} we state, without proof, all intermediate results needed to prove that
the approximating processes  converge in distribution to a billiard process with Markovian 
reflections. All these results and Theorem \ref{mainTheorem}, our main result, are proved  in Section \ref{Sect4}.

\section{Discrete approximations}
\label{Sect1}
\subsection{Discrete-space Markov processes with memory}\label{ch3:section:MarkovProcessesMemory}

We will review the context as well as the main result from  \cite{BW1} in this subsection. Let $(X, L)$ be 
a continuous time Markov processes with state space $\calD_n\times\R^d$,
where  $d\geq 1$ and $\calD_n=\{0,1,\dots,n\}$.
We associate a vector $v_j \in \R^d$ to each $j \in \calD_n$, and define 
\[
L_j(t) = \Leb(s\in[0, t] : X(s) = j)
\] as the
time $X$ has spent at location $j$ until time $t.$ The ``memory'' process is defined as
\[
L(t) = \sum_{j \in \calD_n}v_jL_j(t).
\] 
Functions
\[
a_{ij}(\ell) : \R^d \to \R
\] 
 govern the intensity of  transitions of $X$ from $i$ to $j$. In other words, conditional on $X(t_0) = i$ and $L(t_0) = \ell$, the intensity of  jumps of $X$  from $i$ to $j$ is $a_{ij}(\ell + [t-t_0]v_i)$ for $t \geq t_0$, until $X$ jumps away from $i$. More precisely, the evolution of the process can be
described as follows. Let $(E_k^j)_{j\in \calD_n ,\;k\geq 0}$ be a family of i.i.d.{} exponential random variables with 
parameter one and let $(T_i)_{i\geq 0}$ be the sequence of times when $X$
changes its position, with $T_0=0$. Assuming that the process is defined up to time $T_i$, we recursively define
\begin{align}\notag
T_{i+1}^j&= \inf\left(t>T_i: \int_{T_i}^t a_{X(T_i)j}\big(L(T_i)+ v_{X(T_i)}(s-T_i)\big)ds\geq E_i^j\right),\\
T_{i+1}&=\min_{j\in \calD_n} T_{i+1}^j,\label{o11.1}
\end{align}  
with the convention that $\inf\emptyset =\infty$. Then, set
\begin{align}\label{Evol_Interjumps}
L(s)&=L(T_i)+ v_{X(T_i)}(s-T_i),& \text{ for } s\in [T_i, T_{i+1}],\\
X(s)&= X(T_i),& \text{ for } s\in [T_i, T_{i+1}),\notag\\
X(T_{i+1})&= \argmin(T_{i+1}^j: j\in \calD_n).&\notag
\end{align}
Note that
\[
\prob(T_{ i+1}^j > t + T_i \mid X(T_i) = k, L(T_i) = \ell) = \exp\left(-\int_0^ta_{kj}(\ell + sv_k)ds\right),
\]
for all $t > 0$.  The pair $(X, L)$ is a strong Markov process with 
infinitesimal generator
\begin{align*}
\mc{A}f(j,\ell)= \langle v_j,\nabla_\ell f(j,\ell)\rangle +\sum_{i\neq j}a_{ji}(\ell)\big(f(i,\ell)-f(j,\ell)\big)
\end{align*}
for $f:\calD_n\times \R^d \rightarrow \R$ of sufficient smoothness. 
It is assumed in \cite{BW1} that $(X, L)$ is irreducible in the sense that there are $j_0 \in \calD_n$ and a non-empty open set $U\subset \R^d$  such that 
\[
\prob((X(t), L(t)) \in \{j_0\} \times U \mid X(0) = i, L(0) = \ell) > 0,
\]
for every $(i, \ell) \in \calD_n\times \R^d$ and some $t>0$ (depending on $(i, \ell)$).

\begin{remark}
See \cite[Chap. 2]{Bremaud} for a formal definition and characterization of doubly-stochastic jump processes such as $X$. Note that 
the stochastic jump intensity of $X$ is adapted to the right continuous filtration generated by $X$.
\end{remark}

 Let $\mc{U}(\calD_n)$ denote  the uniform distribution on $\calD_n$ and let $\calN_d$ be the $d$-dimensional standard normal distribution.
Our model and arguments will be based on the following result.

\begin{theorem}\label{Theorem:BurdzyWhite}\cite[Cor. 2.3]{BW1}
The stationary distribution for $(X, L)$ is $\calU(\calD_n)\times \calN_d$ if and only if
\begin{align}\label{eq1}
v_j\cdot \ell + \sum_{i \in \calD_n}a_{ij}(\ell) - \sum_{i\in \calD_n} a_{ji}(\ell) = 0,
\end{align}
for all $j\in \calD_n$ and $\ell \in \R^d$.
\end{theorem}

\begin{remark}%\label{Theorem:BurdzyWhite}
Heuristically,
condition \eqref{eq1} can be represented  as
\begin{align}\label{s21.1}
 v_j\cdot \ell + (\text{flow into $j$}) - (\text{flow out of $j$}) = 0,
\end{align}
for all $j\in \calD_n$ and $\ell \in \R^d$.
\end{remark}

In the next two sections we will specify $v_j$ and $a_{ij}$ that will give rise to a billiard process with Markovian reflections.

\subsection{Approximating processes}

We will consider a sequence of processes $(X_n, L_n)$, $n\geq 2$, defined
as in Section \ref{ch3:section:MarkovProcessesMemory}, with the state space $\calD_n \times \R$, where
$\calD_n = \{ 0, 1,\dots, n\}$.
We will always assume that $a_{ij}(\ell) =0 $ whenever $|i-j| \ne 1$ (we will suppress $n$ in the notation $a_{ij}(\ell) $). Hence $X_n$ will be a nearest neighbor random walk with random transition probabilities.

Heuristically, $\calD_n$ should be thought of as a discretization of $[0,1]$. We chose to label the elements of $\calD_n$ as $\{ 0, 1,\dots, n\}$ rather than $\{ 0, 1/n,2/n,\dots, (n-1)/n,1\}$ for typographical reasons.
The state space $\calD_n$ will have two ``boundary regions'' 
$\prt\calD_n^-:= \{0,1,\dots,N_0(n)\}$ and $\prt\calD_n^+:=\{ n - N_1(n) , \dots, n-1, n\}$.

\begin{assumption}\label{o8.1}
The following are (some of) our standing assumptions.
\begin{enumerate}[label=(\roman*)]
\item $0 \leq N_0(n), N_1(n)< n/2$, for
$ n\geq2$,
\item
$\lim_{n\rightarrow \infty} N_k(n)/n=0$, for $ k=0,1$,

\item $v_j(n) = 0$ if and only if $j\in \{N_0(n)+1, n - N_1(n) -1\}$,

\item $v_j(n) > 0$ if $j \in \prt\calD_n^-$,

\item $v_j(n) < 0$ if $j \in \prt\calD_n^+$.

\end{enumerate}
\end{assumption}

Assumption \ref{o8.1} (iii) means that the memory process $L_n$ is not affected when $X_n$ is outside the boundary regions
$\prt\calD_n^-$ and $\prt\calD_n^+$. We will choose $a_{ij}(\ell)$ so that, as a consequence of Assumption \ref{o8.1} (iii), the ``drift'' of $X_n$ will not be affected outside the boundary regions.

\begin{definition} The boundary $0$ (resp. $n$) is said to be \textit{hard} if and only if $N_0(n)=0$ (resp. $N_1(n)=0$); otherwise it is said to be \textit{soft}. The boundaries are said to be \textit{noiseless} if  $a_{i,j}(\ell)>0$ if and only if $(j-i)\ell >0$ for all $i\in\prt\calD_n^-\cup\prt\calD_n^+$; otherwise they are said to be \textit{noisy}. 
\end{definition}
Our motivation for this terminology is the following. Since the process $(X_n,L_n)$ is supposed to approximate a billiard process, its ``velocity component'' $L_n$ should change
only if $X_n$ is in one of the boundary regions $\prt\calD_n^-$ and $\prt\calD_n^+$.  The term  \textit{``soft''} refers to the idea that the repulsive effect is felt away from the boundaries 0 and $n$, while \textit{``hard''} designates the opposite case. 

The term \textit{``noisy''} refers to the idea that the ``drift'' of the particle does not determine the direction of the motion in a deterministic way---the particle can go in both directions with 
positive probabilities.

\subsubsection{Noiseless case}
We will discuss only the lower boundary region $\prt\calD_n^-$. Implicitly, we make analogous assumptions  for the  region $\prt\calD_n^+$; therefore, analogous results hold for the upper boundary region.

Recall that we have assumed $a_{ij}(\ell) =0 $ whenever $|i-j| \ne 1$.
In the noiseless case we will assume that for all $n, \ell $ and $i\in\calD_n\setminus\{n\}$, and some $c_i(n) >0$, the transition rates have the form
\begin{align}\label{s21.7}
a_{i, i + 1}(\ell) &= 
\begin{cases}
c_i(n)\ell,&\text{if  }\ell\geq 0,\\
0 & \text{otherwise},
\end{cases}\\
a_{i + 1, i}(\ell) &=
\begin{cases}
c_i(n)(- \ell ), & \text{if  }\ell\leq 0,\\
0 & \text{otherwise}.
\end{cases}\label{s21.8}
\end{align}
Let
\begin{align}\label{o8.2}
c_i(n) = n, \qquad \text{  for  } N_0(n)\leq i \leq n - N_1(n).
\end{align}
When $\ell < 0$ we have the following
schematic representation of the  probability mass flow into and out of  $i\in\prt\calD_n^-\setminus\{0\}$,  (see \eqref{s21.1}), 
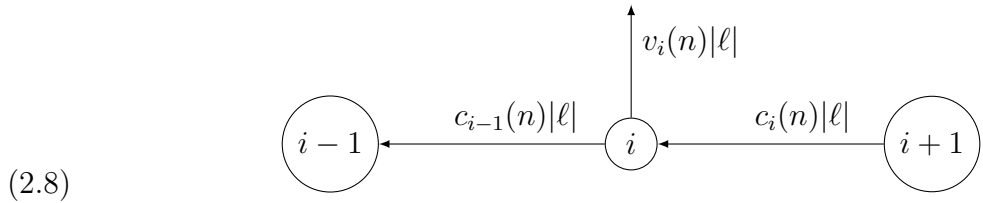
\begin{center}
\begin{equation}\label{schema1}
\begin{tikzpicture}
\node[draw,circle,fill=gray!0] (A)at(3,3){$i-1$};
\node[draw,circle,fill=gray!0] (B)at(7,3){$i$};
\node[draw,circle,fill=gray!0] (C)at(11,3){$i+1$};
\node(draw) (D)at(7,5){};
\draw[<-, >= latex] (A)--(B);
\draw (4.5,3) node [above right] {$c_{i-1}(n)|\ell|$};
\draw[<-, >= latex] (B)--(C);
\draw (8.5,3) node [above right] {$c_{i}(n)|\ell|$};
\draw[->, >= latex] (B)--(D);
\draw (7,4) node [above right] {$v_i(n)|\ell|$};
\end{tikzpicture}
\end{equation}
\end{center}
 Assumption \ref{o8.1} (iv) implies that $v_i(n)\ell < 0$. Therefore, the corresponding arrow  shows the  ``outflow'' from $i$.

Following \eqref{eq1} we equate the sum of signed flows to zero, to obtain for $i\in\prt\calD_n^-\setminus\{0\}$,
\begin{align}\label{s21.2}
0 = c_i(n)|\ell| - c_{i - 1}(n)|\ell | - v_i(n)|\ell| \iff v_i(n) + c_{i - 1}(n) = c_i(n).
\end{align}
When $i = 0$, the schematic is the following,
\begin{center}
\begin{equation}\label{s23.3}
\begin{tikzpicture}
\node[draw,circle,fill=gray!0] (A)at(3,3){$0$};
\node[draw,circle,fill=gray!0] (B)at(7,3){$1$};
\node(draw) (D)at(3,5){};
\draw[<-, >= latex] (A)--(B);
\draw (4.5,3) node [above right] {$c_{0}(n)|\ell|$};
\draw[->, >= latex] (A)--(D);
\draw (3,4) node [above right] {$v_0(n)|\ell|$};
\end{tikzpicture}
\end{equation}
\end{center}
which yields the formula 
\begin{align}\label{s21.3}
v_0(n) = c_0(n).
\end{align}
We combine \eqref{s21.2}-\eqref{s21.3} to obtain the following system of equations for $c_i(n)$'s and $ v_i(n)$'s,
\begin{align}
\begin{split}
v_{N_0(n)}(n) + c_{N_0(n) - 1}(n) &= c_{N_0(n)}(n) = n,\\
v_{N_0(n) - 1}(n) + c_{N_0(n) - 2}(n) &= c_{N_0(n) - 1},\\
\vdots\\
v_1(n) + c_0(n) &= c_1(n),\\
v_0(n) &= c_0(n).
\label{eq:system}
\end{split}
\end{align}

It follows from \eqref{eq1} and \eqref{s21.7}-\eqref{s21.8} that we obtain the same system of equations \eqref{eq:system} in the case when $\ell > 0$.
It follows easily from \eqref{eq:system} that 
\begin{align}\label{s22.1}
c_k(n) = \sum_{i = 0}^kv_i(n),\quad 0\leq k \leq N_0(n).
\end{align}
In particular,
\begin{align}\label{s22.2}
\sum_{i = 0}^{N_0(n)}v_i(n) = c_{N_0(n)}(n) = n.
\end{align}
In order to analyze the evolution of $L_n$ inside the soft boundaries, we will need the following quantities:   
\begin{align}\label{s28.5}
\lambda_{i}(n) &:= \frac{v_i(n)}{c_i(n)}, \quad i= 0, \dots, N_0(n),\\
\mu_i(n) &:= \frac{v_{i+1}(n)}{c_{i}(n)}, \quad i= 0, \dots, N_0(n) - 1.\label{s28.6}
\end{align}
These are the ratios of the ``memory accumulation rates'' at  sites $i$ and $i+1$  and the jump rate between these two sites (per unit of memory $L_n$); see Fig. \eqref{schema1}.

In view of \eqref{s21.3}, we have $\lambda_0(n)=1$ for all $n$. 

We will use  the following assumptions in some of our arguments.
\begin{enumerate}
\item[$\bold{F1}$:] $\lambda_{i}(n)\neq \lambda_{j}(n)$ for all  $i,j\in \prt \calD_n^-$ such that $j\neq i$.
\item[$\bold{F2}$:] $\mu_{i}(n)\neq \mu_{j}(n)$ for all $i,j= 0, \dots, N_0(n) - 1$ such that $j\neq i$.
\item[$\bold{F'}$: ] $v_j(n) \geq v_{j+1}(n) >0$ for all $j\in \{0,\dots, N_0(n) -1\}$.
\end{enumerate}

We will argue that $\bold{F'}$ implies $\textbf{F1}$-$\textbf{F2}$. We will use \eqref{s22.1}. We have $\lambda_i(n) > \lambda_{i+1}(n)$ if and only if the following equivalent conditions hold,
\begin{align}\label{s30.10}
\frac{v_i(n)}{c_i(n)} > \frac{v_{i+1}(n)}{c_{i+1}(n)}
&\LRa
v_i(n)c_{i+1}(n) > v_{i+1}(n)c_{i}(n) 
\LRa v_i(n)\sum_{j = 0}^{i+1}v_j(n) > v_{i+1}(n)\sum_{j = 0}^{i}v_j(n)\\
&\LRa
v_i(n)v_{i+1}(n) + (v_i(n)-v_{i+1}(n))\sum_{j = 0}^{i}v_j(n) >0. \label{s30.13} 
\end{align}
If $\bold{F'}$ holds then the last inequality is true and, therefore, $\lambda_i(n) > \lambda_{i+1}(n)$. This shows that $\bold{F'}$ implies  $\textbf{F1}$. The calculations showing that $\bold{F'}$ implies  $\textbf{F2}$ are similar:
\begin{align}\label{s30.11}
\frac{v_{i+1}(n)}{c_i(n)} > \frac{v_{i+2}(n)}{c_{i+1}(n)}
&\LRa
v_{i+1}(n)c_{i+1}(n) > v_{i+2}(n)c_{i}(n) \\
&\LRa v_{i+1}(n)\sum_{j = 0}^{i+1}v_j(n) > v_{i+2}(n)\sum_{j = 0}^{i}v_j(n)\\
&\LRa
v_{i+1}(n)v_{i+1}(n) + (v_{i+1}(n)-v_{i+2}(n))\sum_{j = 0}^{i}v_j(n) >0. \label{s30.12} 
\end{align}

In the case when $N_0(n)=N$ for all $n$, we will make the following assumption.
\begin{enumerate}
\item[$\bold{F3}$:] $\lim_{n\rightarrow \infty} v_j(n)/n=\beta_j >0$ for all $j=0,\cdots, N_0(n)=N$. 
\end{enumerate}

\begin{remark}\label{s30.20}
(i) If $\bold{F3}$ holds then $\sum_{j=0}^{N}\beta_j = 1$ because of \eqref{s22.2}. 

(ii) It is easy to check that if \textbf{F3} is true then the limits 
\begin{align}\label{o2.1}
\lambda_i :=\lim_{n\rightarrow\infty}\lambda_i(n) &=\frac{\beta_i}{\sum_{j=0}^i \beta_j},\quad i = 0, \dots, N_0(n),\\
\mu_i :=\lim_{n\rightarrow\infty}\mu_i(n) &=\frac{\beta_{i+1}}{\sum_{j=0}^i \beta_j},\quad i = 0, \dots, N_0(n)-1,\label{o2.2}
\end{align}
exist. 

(iii)
Assumption $\bold{F3}$ and \eqref{s22.1} imply that the limits  $\lim_{n\rightarrow \infty} c_j(n)/n=c_j >0$ exist for all $j=0,\cdots, N_0(n)=N$.
By assumptions $\bold{F'}$ and $\bold{F3}$, we have $0<\beta_{j+1}\leq \beta_j$ for all $j\in \{0,\cdots, N_0(n)-1\}$. The calculations \eqref{s30.11}-\eqref{s30.12} can be repeated with $v_j(n)$ replaced with $\beta_j$ for $j=i+1, i+2$, and $c_j(n)$ replaced with $c_j$ for $j=i, i+1$. With this substitution, the conclusion of that calculation is that $\mu_i > \mu_{i+1} >0$. 
\end{remark}

If $N_0(n)$ grows to infinity with $n$,  instead of \textbf{F3}, we will adopt the following assumptions. First, let
\begin{align}\label{o2.3}
\lambda'_j(n) &= 
\begin{cases}
		\lambda_{j}(n) & \text{if } j\leq N_0(n),\\ 
		0 & \text{otherwise, }
	\end{cases}\\
\mu'_j(n) &= 
\begin{cases}
		\mu_{j}(n) & \text{if } j\leq N_0(n),\\ 
		0 & \text{otherwise. }
	\end{cases}
\label{o2.4}
\end{align}
The new assumptions are
\begin{enumerate}
\item[$\bold{G1}$:] $(\mu'_j(n))_{j\geq 0}$ converges in $\ell^1$ to $(\mu_j)_{j\geq 0}$ as $n\to\infty$.
\item[$\bold{G2}$:] $(\lambda'_j(n))_{j\geq 0}$ converges in $\ell^1$ to $(\lambda_j)_{j\geq 0}$ as $n\to\infty$.
\end{enumerate}

\subsubsection{Noisy case}

In this case, we will give explicit formulas only in the case  $N_0(n)=N_1(n)=1$.
In the general case the formulas are too complicated to be useful or informative.

Recall that we have assumed that $a_{ij}(\ell) =0 $ whenever $|i-j| \ne 1$.
In the noisy case we will assume that for all $n, \ell $ and $i\in\calD_n\setminus\{n\}$, and some $b_i(n),c_i(n) >0$, the transition rates have the form,
\begin{align}\label{o3.1}
a_{i, i + 1}(\ell) &= 
\begin{cases}
c_i(n)\ell,&\text{if  }\ell\geq 0,\\
b_{i+1}(n)(-\ell) & \text{if  }\ell< 0,
\end{cases}\\
a_{i + 1, i}(\ell) &=
\begin{cases}
c_i(n)(- \ell ), & \text{if  }\ell\leq 0,\\
b_{i+1}(n)\ell & \text{if  }\ell> 0.
\end{cases}\label{o3.2}
\end{align}
For $i \in \{1,\dots,  n - 2\}$ we set $c_i(n) = n$ and $b_{i+1}(n)=0$. 

By symmetry of our model, we can focus on the lower boundary $\prt \calD_n^-$. Updating the noiseless schematics in \eqref{schema1} and \eqref{s23.3}, we obtain in the noisy case, when $\ell < 0$,
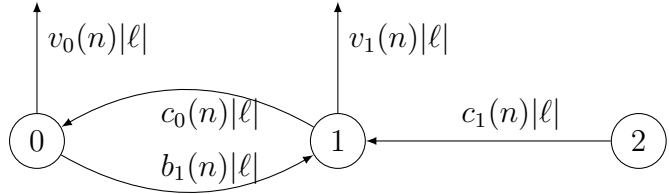
\begin{center}
\begin{tikzpicture}
\node[draw,circle,fill=gray!0] (A)at(3,3){$0$};
\node[draw,circle,fill=gray!0] (B)at(7,3){$1$};
\node[draw,circle,fill=gray!0] (C)at(11,3){$2$};
\node(draw) (D)at(7,5){};
\node(draw) (E)at(3,5){};
\draw[<-, >= latex] (A)to[bend left](B);
\draw (4.5,3) node [above right] {$c_{0}(n)|\ell|$};
\draw[->, >= latex] (A)to[bend right](B);
\draw (4.5,3) node [below right] {$b_{1}(n)|\ell|$};
\draw[<-, >= latex] (B)--(C);
\draw (8.5,3) node [above right] {$c_{1}(n)|\ell|$};
\draw[->, >= latex] (B)--(D);
\draw (7,4) node [above right] {$v_1(n)|\ell|$};
\draw[->, >= latex] (A)--(E);
\draw (3,4) node [above right] {$v_0(n)|\ell|$};
\end{tikzpicture}
\end{center}
In this schematics $v_i \geq 0$,  the values adjacent to the arrows indicate the magnitude of the incoming or outgoing flow,
and the direction designates the sign. When $\ell >0$, the schematics remain valid except that the direction of the arrows should be reversed.
Following \eqref{eq1}-\eqref{s21.1} we equate the sum of signed flows to zero. With the convention $c_{-1}(n)=c_{n}(n)=b_0(n)= b_{n+1}(n)=0$, we get for all $i\in\calD_n$,
\begin{align*}
&0 = c_i(n)|\ell| + b_i(n)|\ell| - c_{i - 1}(n)|\ell| - v_i(n)|\ell|- b_{i+1}(n)|\ell| \\
\iff &b_{i+1}(n)+ v_i(n) + c_{i - 1}(n) = b_i(n) + c_i(n).
\end{align*}
Hence, 
\begin{align*}
c_0(n) &= v_0(n) + b_1(n),\\
c_1(n) &= c_2(n)= v_0(n)+ v_1(n) =n.
\end{align*}  
In the noisy case, we will use the following assumption.  
\begin{enumerate}
\item[$\bold{K}$: ] $\lim_{n\rightarrow\infty}
\displaystyle \frac{b_1(n)}{n}= \vartheta_1 \geq 0$.
\end{enumerate}

\section{Convergence of approximations}\label{Sect2}

This section contains  intermediate results needed to prove Theorem \ref{mainTheorem}. Some of them may have independent interest. All proofs will be postponed to Section \ref{Sect4}.

\subsection{Noiseless case}

The discussion of the noiseless case will be further subdivided into two cases, those of the hard boundary and soft boundary.

\subsubsection{Hard boundary}
Recall that ``hard boundary'' refers to the case $N_0(n) = 0$. Hence, $s\mapsto L_n(s)$ changes only if $X_n(s)\in \{0, n\}$.
This implies that if $X_n$ jumps to 0 at some time $t>0$, we must have $L_n(t) < 0$. Our transition rates are chosen so that $X_n$ cannot leave 0 until $L_n$ changes sign to positive. Thus, let us suppose that $(X_n(0),L_n(0))=(0,0)$. Recall our notation from \eqref{o11.1} and the assumption that $c_0(n)=v_{0}(n)=n$. Let $E_1$ be an exponential random variable with mean 1.
We have
\begin{align*}
T_1&=\inf\left(t\geq 0 : \int_0^t a_{01}(sv_{0})ds\geq E_1\right)
= \inf\left(t\geq 0 : \int_0^t n^2 s ds\geq E_1\right)
=\frac{\sqrt{2E_1}}{n}.
\end{align*} 
Hence,
if $\ell<0$ and $(X_n(0),L_n(0))=(0,\ell)$ then the distribution of $L_n(T_1)= v_0(n)T_1$ is the same as  the distribution of $\sqrt{2E_1}$.
Therefore,
\begin{align*}
\mathbb{P}(L_n(T_1)>r )&= \mathbb{P}(E_1>r^2/2)
=\exp(-r^2/2).
\end{align*}
Consequently, the density of $L_n(T_1)$ is
$r\exp(-r^2/2)$ for $r>0$. This is the density of what is called the Rayleigh distribution with parameter 1.

The unique feature of the hard boundary reflection is that the distribution of the ``velocity'' just after the reflection  depends  neither on the incoming velocity nor on  $n$.

\subsubsection{Soft boundary}\label{soft_boundary}

In the soft boundary case, the evolution is more interesting than in the hard boundary case. 
At the moment when the process $X_n$ enters the lower boundary layer $\prt \calD_n^-$, its ``velocity'' $L_n$ must be negative. The particle $X_n$ will continue to transition downward until $L_n$ changes sign or $X_n$ reaches 0.  Consequently, we must determine the distribution of the level at which the velocity $L_n$ changes sign. Once the velocity becomes positive, it increases
until $X_n$ exits $\prt \calD_n^-$.

 Let 
\begin{align}\label{s30.21}
\mc{T}_n &= \inf\{t\geq 0: L_n(t) \geq 0\}, \\
 G_n &= X_n(\mc{T}_n), \label{s30.22}\\
\mc{U}_n &= \inf\{t\geq \mc{T}_n: X_n(t) \notin \prt \calD_n^-\}, \label{s30.23}\\
\mc{V}_n(\ell) &= \mc{L}(L_n(\mc{U}_n)\mid X_n(0) =N_0(n),L_n(0) = \ell),
\qquad \ell<0,
\label{o5.3}\\
\mc{V}^+_n(\ell) &= \mc{L}(L_n(\mc{U}_n)\mid X_n(0) =N_1(n),L_n(0) = \ell),
\qquad \ell>0,
\label{o5.4}\\
p_j(n, \ell) &=
\mathbb{P}\left( G_n =j \mid X_n(0) =N_0(n), L_n(0) = \ell \right).\label{s24.1}
\end{align}

\begin{proposition}\label{LevelDistrib} Assume  either $\bold{F2}$ or $\bold{F'}$. For $\ell <0$,
\begin{align}\label{s30.1}
&p_{N_0(n)}(n, \ell)=\exp\left(-\frac{\ell^2}{2\mu_{N_0(n)-1}(n)}\right),\\
&p_k(n, \ell) 
=\left(\prod_{j=k}^{N_0(n)-1} \frac1{\mu_j(n)} \right)
\sum_{j=k-1}^{N_0(n)-1} \left(\exp\left(-\frac{\ell^2}{2 \mu_j(n)}\right) 
\prod_{\substack{i=k-1\\i \ne j}}^{N_0(n)-1}
\frac 1{1/\mu_i(n) - 1/\mu_j(n)}
\right),\label{o1.2}\\
& \qquad\qquad \text{  for  } 0 < k <  N_0(n),\notag\\
&p_0(n, \ell)  = 1-\sum_{k=1}^{N_0(n)}p_k(n, \ell) .\label{s30.3}
\end{align}
\end{proposition}

\begin{proposition}\label{CondExitVel}  Assume  either $\bold{F1}$ or $\bold{F'}$. Given $k\in \prt \calD_n^-$ and $\ell<0$, and conditional on $X_n(0) =N_0(n)$, $L_n(0) = \ell$, and $G_n = k $, the distribution of  $L_n(\mc{U}_n)$ is the same as that of
\begin{align}\label{s29.3}
\left(2\sum_{j=k}^{N_0(n)} \lambda_j(n) E_j\right)^{1/2},
\end{align}  
where $E_k, \cdots, E_{N_0(n)}$ are i.i.d. exponential random variables with mean 1.  The density  of this random variable is equal to
\begin{align}\label{s29.4}
f_{k,n}(r) &:=
r\left(\prod_{j=k}^{N_0(n)}\frac1{\lambda_j(n)} \right)
\sum_{j=k}^{N_0(n)} \left(\exp\left(-\frac{r^2}{2\lambda_j(n)}\right) 
\prod_{\substack{i=k\\i \ne j}}^{N_0(n)}
\frac 1 {1/\lambda_i(n) - 1/\lambda_j(n)}
\right).
\end{align}
\end{proposition} 

The following corollary follows easily from Propositions \ref{LevelDistrib} and \ref{CondExitVel} and the strong Markov property applied at $\mc{T}_n$, so we will not supply a formal proof.

\begin{corollary}\label{o2.7}
Assume  $\bold{F'}$. If $ \ell<0$, $\mc{V}_n(\ell)$ is the same as the distribution of
\begin{align}\label{s24.5}
 \sum_{k=0}^{N_0(n)}\left(2\sum_{j=k}^{N_0(n)}\lambda_j(n)E_j\right)^{1/2} \bone_{Z=k},
\end{align}
where $E_j$'s are are i.i.d. exponential random variables with mean 1 and $Z$ is an independent random variable with 
$\P(Z=j) = p_j(n,\ell)$ for $j\in \prt \calD_n^-$, where $p_j(n,\ell)$ are as in \eqref{s30.1}-\eqref{s30.3}.
\end{corollary}

\begin{theorem}\label{ThinBoundaries}
Suppose that $\ell_n < 0$ for $n\geq 1$ and  $\lim_{n\rightarrow \infty}\ell_n =\ell < 0$. Assume  $\bold{F1}$-$\bold{F2}$ or $\bold{F'}$, and $\bold{F3}$. Suppose that $E_1, E_2, \dots$ are i.i.d. exponential random variables with mean 1. 

(i) Assume that $N_0(n) = N< \infty$ for all $n$. Then for every $k\in \prt \calD_n^-$, the following limit exists,
\begin{align}\label{s30.4}
p_k(\ell) := \lim_{n\rightarrow\infty} p_k(n, \ell_n).
\end{align} 

(ii) Assume that $N_0(n) = N< \infty$. Then, when $n\to \infty$, $\mc{V}_n(\ell_n)$ converge  to the distribution of
\begin{align*}
\sum_{j=0}^N \left(2\sum_{i=j}^N\lambda_i E_i\right)^{1/2}
\bone_{Z(\ell)=j},
\end{align*}
where $Z(\ell)$ is a random variable
with values in $\{0,1, \dots, N\}$, independent of $E_j$'s and such that
 $\mathbb{P}(Z(\ell)=j)=p_j(\ell)$, $ j=0,\dots, N$. The values of $p_j(\ell)$, $j=0,\dots, N$ are given by \eqref{s16.1}-\eqref{s23.5}.

(iii) Assume that $N_0(n) = 1$. Then, when $n\to \infty$, $\mc{V}_n(\ell_n)$ converge  to the distribution of
$$\sqrt{2\left(E_0 + \beta_1 E_1\right)}\bone_{Z(\ell)=0}+ \sqrt{2\beta_1 E_1}\bone_{Z(\ell)=1},$$ 
where $Z(\ell)$ is a random variable independent of the collection of $E_j$'s such that $\mathbb{P}(Z(\ell)=j)= p_j(\ell),\; j=0,1$. The values of $p_0$ and $p_1$ are given by \eqref{s23.4}.

\end{theorem}

\begin{remark} We presented the case $N_0(n)=1$ in Theorem \ref{ThinBoundaries}, in addition to  the general case $N_0(n)=N$, so that Theorem \ref{ThinBoundaries} (iii) may be directly compared to Theorem \ref{Cvge_Veloc}, its counterpart in the case of noisy soft boundary. 
\end{remark}

\begin{proposition}\label{CVGE_probs} Assume  $\bold{G1}$. Suppose that $\ell_n < 0$ for $n\geq 1$,  $\lim_{n\rightarrow \infty}\ell_n =\ell < 0$, and $\lim_{n\to \infty } N_0(n) =\infty$. Then, for every $k\geq 0$, the following limit exists,
\begin{align}\label{s23.6}
p_k(\ell) := \lim_{n\rightarrow\infty} p_k(n, \ell_n).
\end{align} 
\end{proposition}

When $N_0(n) \to \infty$ as $n\to\infty$, the counterpart of Theorem \ref{ThinBoundaries} is the following. 

\begin{theorem}\label{Infinite_layers}
Assume  $\bold{G1}$-$\bold{G2}$. Suppose that $\ell_n < 0$ for $n\geq 1$, $\lim_{n\rightarrow \infty}\ell_n =\ell < 0$, and $\lim_{n\to \infty } N_0(n) =\infty$. 
 Then, when $n\to \infty$, $\mc{V}_n(\ell_n)$ converge  to the distribution of
\begin{align*}
\sum_{j=0}^\infty \left(2\sum_{i=j}^\infty\lambda_i E_i\right)^{1/2}
\bone_{Z(\ell)=j},
\end{align*}
where $E_1, E_2, \dots$ are i.i.d.\ exponential random variables with mean 1 and $Z(\ell)\geq 0$ is independent of the collection of $E_j$'s such that
 $\mathbb{P}(Z(\ell)=j)=p_j(\ell)$, $ j\geq0$. The probabilities $p_j(\ell)$ are defined in \eqref{s23.6}.
\end{theorem} 

\subsection{Noisy case} 

The following result is a noisy counterpart of Proposition \ref{LevelDistrib}. 

\begin{proposition}\label{Layer_Reverse} Assume that  $\ell<0$. 
Let $\beta_1(n) = c_0(n)/v_1(n)$ and $\beta_2(n) = b_1(n)/v_0(n)$.

(i)
If $\beta_1(n)\neq \beta_2(n)$ then,
\begin{align*}
 p_1(n, \ell)
&=\sum_{k\geq 0}\Bigg(\beta_1(n)^k\beta_2(n)^k \exp\left(-\beta_1(n)\ell^2/2\right)\times\\
&\qquad\times\int_0^{\ell^2/2}\Bigg[\sum_{i=1}^2\sum_{j=1}^k \frac{(-1)^{k-j}}{(j-1)!}u^{j-1}\exp(-(\beta_i(n)-\beta_1(n))u)\times\\
&\qquad\qquad\qquad\times \binom{2k-j-1}{k-j}\big(\beta_{3-i}(n)-\beta_i(n)\big)^{-(2k-j)}\Bigg]du\Bigg).
\end{align*}

(ii)
If  $\beta_1(n)= \beta_2(n)$ then,
\begin{align*}
 p_1(n, \ell)
=\sum_{k\geq 0} \frac{(\beta_1(n)\ell^2)^{2k} }{2^{2k}(2k)!}\exp\left(-\beta_1(n) \ell^2/2\right).
\end{align*}
\end{proposition}

\begin{corollary}\label{Limit_Layer_reverse} Assume  $\bold{F3}$ and $\bold{K}$.
Suppose that $\ell_n < 0$ for $n\geq 1$ and  $\lim_{n\rightarrow \infty}\ell_n =\ell < 0$. Then $p_1(\ell) = \lim_{n\rightarrow\infty}p_1(n,\ell_n) $ exists.  
\end{corollary}

The term ``geometric distribution'' may refer to either of two closely related distributions. In this article, a random variable $R$ will be called geometric with parameter $p$  if $\P(R = j) = (1-p)^{j-1}p$ for $j=1,2,\dots$.

We have the following analogue of Proposition \ref{CondExitVel}.

\begin{proposition}\label{Exit_Velo} 
If $\ell<0$ then $\mc{V}_n(\ell)$ is the distribution of
\begin{align}\label{s24.3}
 \left(2\frac{v_0(n)}{c_0(n)} E +  S_n\right)^{1/2}\bone_{Z=0}+ S_n^{1/2} \; \bone_{Z=1} ,
\end{align}
where
\begin{align}\label{SnL}
S_n = 2\frac{v_1(n)}{c_1(n)}E'+\sum_{k=1}^{J(n) -1} 2\frac{v_1(n)}{b_1(n)}E_j'+\sum_{k=1}^{J(n) -1} 2\frac{v_0(n)}{c_0(n)}E_k''.
\end{align}
The random variables $E$, $E'$, $(E_k')_k$, $(E_k'')_k$ are i.i.d\ exponential with mean 1, $J(n)$ is geometric  with parameter $c_1(n)/(c_1(n)+b_1(n)) $, and $Z$ takes values 0 or 1
and satisfies $\P(Z=1) = p_1(n,\ell)$, where $p_1(n,\ell)$ is given in Proposition \ref{Layer_Reverse}. All of these random variables are assumed to be independent.  
\end{proposition}

\begin{theorem}\label{Cvge_Veloc}
Assume  $\bold{F3}$ and $\bold{K}$. Suppose that $\ell_n < 0$ for $n\geq 1$ and  $\lim_{n\rightarrow \infty}\ell_n =\ell < 0$.
 Then there exist constants $\gamma_0,\gamma_1,\gamma_2 \in (0,+\infty)$ and $s\in[0,1)$ such that
\begin{align}\label{o5.1}
\gamma_0 &= \lim_{n\rightarrow\infty}\frac{2v_0(n)}{c_0(n)}, \qquad
\gamma_1 = \lim_{n\rightarrow\infty}\frac{2v_1(n)}{c_1(n)},
\qquad \gamma_2 = \lim_{n\rightarrow\infty}\frac{2v_1(n)}{b_1(n)} ,\\
 \gamma_3&= \lim_{n\rightarrow\infty}\frac{c_1(n)}{c_1(n)+b_1(n)}. \label{o5.2}
\end{align}
Moreover, distributions $\mc{V}_n(\ell_n)$ converge to the distribution of
\begin{align}\label{o5.8}
\sqrt{\gamma_0 E + S}\;\bone_{Z(\ell)=0}+ \sqrt{S}\;\bone_{Z(\ell)=1},
\end{align}
where
\begin{align}\label{o5.7}
S = \gamma_1 E' + \sum_{j=1}^{J-1}\left(\gamma_0 E_j' + \gamma_2 E_j''\right).
\end{align}
The random variables $E$, $E'$, $(E_k')_k$, $(E_k'')_k$ are i.i.d. exponential with mean 1.
The distribution of 
$Z(\ell)\in \{0,1\}$ is determined by $\mathbb{P}(Z(\ell)=1)= p_1(\ell)$, with $p_1(\ell)$ defined in Corollary \ref{Limit_Layer_reverse}. The
random variable $J$ is  geometric  with parameter $\gamma_3$.
All  these random variables are independent.  
\end{theorem}

\subsection{Convergence to the billiard process}\label{Sect3}
We will prove, under appropriate assumptions, that the sequence $(X_n/n, L_n)$ converges in distribution to a billiard process with Markovian reflections. Let
\begin{align}\label{o5.10}
t_0(n) & = 0,\\
s_j(n) & = \inf\{t\geq t_j(n): X_n(t) \in \prt \calD_n^- \cup \prt \calD_n^+\} ,\quad j\geq 0,\label{o5.11}\\
t_{j+1}(n) & = \inf\{t\geq s_j(n): X_n(t) \in \calD_n \setminus( \prt \calD_n^- \cup \prt \calD_n^+)\}, \quad j\geq 0.\label{o5.12}
\end{align}
These are successive times when the process $X_n$ enters or leaves the boundaries. 

\begin{proposition}\label{timeOnBoundary} 
 Suppose that the distributions of $ L_n(0)$, $n\geq 1$, are tight.
Make one of the following assumptions.

(i) Consider the noiseless model.
If $N_0(n) = N < \infty$ for all $n$,
 suppose that  $\bold{F1}$-$\bold{F2}$ or $\bold{F'}$, and $\bold{F3}$ hold. If $\lim_{n\rightarrow} N_0(n)=\infty$, assume instead  $\bold{G1}$-$\bold{G2}$. 

(ii) Consider the noisy model and suppose that  $\bold{F1}$-$\bold{F2}$ or $\bold{F'}$, $\bold{F3}$, and $\bold{K}$ hold. Recall that $N_0(n)=1$ for all $n$. 

Then for every $j\geq 0$,  
\begin{align*}
 \lim_{n\rightarrow\infty} t_{j+1}(n)-s_{j}(n) =0,\qquad\text{  in distribution.} 
\end{align*}

\end{proposition}

\begin{definition}
Recall definitions \eqref{o5.3}-\eqref{o5.4} of $\mc{V}_n(\ell)$ and $\mc{V}^+_n(\ell)$. 
 Let
\begin{align}\label{s25.1}
\mc{V}_\infty(\ell) &= \lim_{n \to \infty} \mc{V}_n(\ell),\qquad \ell<0,\\
\mc{V}^+_\infty(\ell)& = \lim_{n \to \infty} \mc{V}^+_n(\ell),\qquad \ell>0,
\label{o5.5}
\end{align}
if the limits exist.
\end{definition}

For assumptions under which the  limit in \eqref{s25.1} exists, see Theorems \ref{ThinBoundaries},
\ref{Infinite_layers} and \ref{Cvge_Veloc}. Analogous results hold for the second limit, by symmetry. The distributions $\mc{V}_\infty(\ell)$ and $\mc{V}^+_\infty(\ell)$ (if the limits in \eqref{s25.1}-\eqref{o5.5} exist) give no mass to 0.

\begin{remark}
It follows from Theorem \ref{ThinBoundaries} that every distribution given in Definition \ref{def:vel} can  be expressed as the limit  for a sequence  $(X_n, L_n)$.
\end{remark}

Suppose that the limits  in \eqref{s25.1}-\eqref{o5.5} exist for every $\ell$ in the corresponding range.
For any $(x_0, \ell_0) \in (0,1)\times \R\setminus\{0\}$, we will define a billiard process $(X,L)$ with Markovian reflections starting from $(x_0, \ell_0)$.

Consider $\ell_0<0$. The case $\ell_0>0$ can be treated in an analogous way.
First we form a Markov chain $\{R_j, j\geq 0\}$ by setting $R_0 = \ell_0$ and giving it the Markovian transition mechanism  
\begin{align}\label{s26.1}
\mc{L} (R_{2j+1} \mid R_{2j} = \ell, R_{2j-1}, \dots, R_0) &= \mc{V}_\infty(\ell), \qquad j\geq 0,\\
\mc{L} (R_{2j+2} \mid R_{2j+1} = \ell, R_{2j}, \dots, R_0) &= \mc{V}^+_\infty(\ell), \qquad j\geq 0.
\label{o5.6}
\end{align}
 We define
the process $(X,L)$ by
\begin{align}\label{s26.2}
u_0&=0,\\
W_0(t) &= x_0+R_0 t, \quad t \geq 0,\\
u_{j+1} &= \inf\{t> u_j: W_{j}(t) \notin (0,1)\}, \quad j \geq 0,\\
W_j(t) & = W_{j-1}(u_j) + R_j (t - u_j), \quad t\geq u_j, j \geq 1,\\
X(t) & = W_j(t), \quad [u_j, u_{j+1}), \quad j\geq 0,\\
L(t) &= R_j,  \quad [u_j, u_{j+1}),\quad j\geq 0.\label{s26.3}
\end{align}
Note that $u_j > u_{j-1}$ for all $j\geq 1$, a.s., because  distributions $\mc{V}_\infty(\ell)$ and $\mc{V}^+_\infty(\ell)$ give no mass to 0 and, therefore, $R_j \ne 0$ for all $j\geq 0$, a.s.

We have constructed a billiard process $(X, L)$ with Markovian reflections on $ [0, \sup_j u_j)$. 

\begin{theorem}\label{convProc}
Assume that $(X_n(0)/n, L_n(0))$ converge in distribution to a pair of random variables $(X(0), L(0))$, as $n\to \infty$. Suppose that $X(0)\in(0,1)$ and $L(0) \ne 0$, a.s.

Make one of the following assumptions.

(i) Consider the noiseless model.
Assume that $N_0(n) = N < \infty$ for all $n$, and
 suppose that  $\bold{F1}$-$\bold{F2}$ or $\bold{F'}$, and $\bold{F3}$ hold.

(ii) Consider the noiseless model.
Assume that $\lim_{n\rightarrow} N_0(n)=\infty$, and suppose that $\bold{G1}$-$\bold{G2}$ hold.

(iii) Assume the noisy model and suppose that  $\bold{F1}$-$\bold{F2}$ or $\bold{F'}$, $\bold{F3}$, and $\bold{K}$ hold. Recall that $N_0(n)=1$ for all $n$.

 Then 

(a) $\sup_j u_j = \infty$, a.s. It follows that $(X,L)$ is defined on $[0,\infty)$.

(b)
$\{(X_n(t)/n, L_n(t)) : t \in [0, T]\}$ converges in distribution
to a billiard process with Markovian reflections $\{(X(t), L(t)) : t\in [0, T]\}$ as $n \to \infty$, for every fixed $T <\infty$. 
The distribution of $(X,L)$ is determined by \eqref{s25.1}-\eqref{s26.3}.

\end{theorem}

\section{Proofs}\label{Sect4}

\begin{remark}\label{o4.1}
We will use the following results  from \cite{MA,JK}. 

(i) Suppose that $E_1, E_2, \dots, E_k$ are i.i.d. exponential random variables with mean 1 and consider $ \alpha _j\in(0,\infty)$, $j=1,\dots, k$, such that $\alpha_i\ne \alpha_j$ for $i\ne j$.
Then, according to \cite[Thm. 2.1]{MA}, the density of $\sum_{j=1}^k \alpha_j E_j $ is equal to
\begin{align}\label{s27.1}
f(r) = \left(\prod_{m=1}^{k}\frac1{\alpha_m} \right)
\sum_{j=1}^k \left(\exp(-r / \alpha_j) 
\prod_{\substack{i=1\\i \ne j}}^{k}
\frac1{1/\alpha_i -1/ \alpha_j}
\right).
\end{align}

(ii) Suppose that $z_1, z_2, \dots , z_m$ are distinct complex numbers and $z\ne z_j$ for all $j$. Then, according to \cite[(2.4)]{MA},
\begin{align}\label{o1.1}
\frac 1 {\prod_{j=1}^m (z_j-z)} = \sum_{i=1}^m \ds\frac 1 {(z_i-z)\ds\prod_{\substack{j=1\\j \ne i}}^m (z_j-z_i)}.
\end{align}

(iii) Consider the sum $S$ of $k_1+k_2 + \dots+ k_r$ independent random variables with exponential distributions. 
Suppose that exactly $k_j$ of these 
random variables have mean $\alpha_j$, for $j=1,\dots, r$. Assume that $\alpha_i\ne \alpha_j$ for $i\ne j$. According to \cite[Thm. 1]{JK} (see also \cite[Thm. 4.1]{MA}), the density $f_S(u)$ of $S$ is given by the following formula, for $u>0$,
\begin{align*}
\sum_{i=1}^r \frac 1{\alpha_i^{k_i}}
 \exp\left(-u / \alpha_i\right)
\sum_{j=1}^{k_i} \frac{(-1)^{k_i-j}}{(j-1)!}
u^{j-1}
\sum_{\substack{m_1+m_2+\dots+m_r=k_i-j\\m_i =0}}
\prod_{\substack{l=1\\l\ne i }}^r
\binom{k_l + m_l -1 }{m_l}
\frac{\alpha_l^{-k_l}}{(\alpha_l^{-1} - \alpha_i^{-1})^{k_l+m_l}}.
\end{align*}

\end{remark}

\begin{proof}[Proof of Proposition \ref{LevelDistrib}]
 
The proof has two steps. First we compute the remaining memory $L_n$ when the particle jumps from site $i$ to site $i-1$. This will allow us to compute the distribution of  $G_n$ in the second part of the proof.

Consider a family $E_j$, $j \geq 1$, of i.i.d. exponential random variables with mean 1. We can represent $(X_n,L_n)$ as follows. 
Suppose that $X_n(0) = N_0(n)$ and $L_n(0)<0$.
For $0 \leq i \leq N_0(n)$ let
\begin{align}
\tau_0 &=0,\label{s28.1}\\
\tau_i &= \inf\left\{t > \tau_{i - 1} : \int_{\tau_{i - 1}}^ta_{N_0(n)-i+1,N_0(n)- i}(L_n(s))ds > E_i\right\}, \qquad i\geq 1,\label{s28.2}\\
\Delta \tau_i &= \tau_{i+ 1} - \tau_{i}, \qquad i\geq 0.\label{s28.3}
\end{align}
The following remarks apply as long as $L_n$ stays negative.
The $\tau_i$'s are the times when $X_n$ jumps from  $N_0(n)-i+1$ to $N_0(n)-i $. The amount of time that the process spends at  $N_0(n)-i$ is represented by $\Delta \tau_i$. It follows that $v_{N_0(n)-i}(n)\Delta \tau_i$
is the amount of memory (i.e., the increment of $L_n$) that is accumulated at $N_0(n)-i$. If the process $X_n$ arrives at 
$N_0(n)-i$ with $L_n(\tau_{i}) = \ell<0$  then $X_n$ will leave $N_0(n)-i$ for $N_0(n)-i - 1$ at time $\tau_{i + 1}$ with the memory process $L_n$ taking the value
\begin{align}\label{s28.4}
L_n(\tau_{i + 1}) = \ell + v_{N_0(n)-i}(n)\Delta \tau_i 
=-(\vert \ell\vert-v_{N_0(n)-i}(n)\Delta \tau_i),
\end{align}
provided this quantity is
negative. 

By definition, $X_n$ arrives at $N_0(n) - i$ at time $\tau_{i}$. 
Recall our choice for $a_{i,i+1}(\ell)$ from \eqref{s21.7}-\eqref{s21.8}. It follows from this and \eqref{s28.1}-\eqref{s28.3} that if
 $L_n(\tau_{i}
) = \ell<0$ then
$\Delta \tau_{i}$ is the smallest $t>0$ such that
\begin{align*}
\int_0^tc_{N_0(n) - i - 1}(n)(|\ell| - v_{N_0(n) - i}(n)s)ds > E_{i+1}.
\end{align*}
We set both sides  equal to each other and solve the resulting equation for $t$ as follows. First, we have
\begin{align*}
|\ell| t - v_{N_0(n) - i}(n)t^2/2 = E_{i+1}/c_{N_0(n) - i - 1}(n),
\end{align*}
and then we find zeros using the quadratic formula,
\begin{align}\label{Time_Spent}
\frac{|\ell|}{v_{N_0(n) - i}(n)} \pm \frac{\sqrt{\ell^2 - 2(v_{N_0(n) - i}(n)E_{i+1}/c_{N_0(n) - i - 1}(n))}}{v_{N_0(n) - i}(n)}.
\end{align}
Provided $2(v_{N_0(n) - i}(n)E_{i+1}/c_{N_0(n) - i - 1}(n))$ is sufficiently small, the first
nonnegative zero is
\begin{align}\label{Time_Spent2}
\frac{|\ell|}{v_{N_0(n) - i}(n)} - \frac{\sqrt{\ell^2 - 2(v_{N_0(n) - i}(n)E_{i+1}/c_{N_0(n) - i - 1}(n))}}{v_{N_0(n) - i}(n)}.
\end{align}
If we take this as the value of $\Delta \tau_{i}$ and combine this formula with \eqref{s28.4}, we obtain,
\begin{align*}
L_n(\tau_{i+1})=-\left(\ell^2-2\frac{v_{N_0(n)-i}(n)E_{i+1}}{c_{N_0(n)-i-1}(n)}\right)^{1/2}.
\end{align*}
It follows from this and \eqref{s28.6} that
\begin{align*}%\label{MemoryChange}
L_n^2(\tau_{i+1})&=L_n^2(\tau_i)-2\frac{v_{N_0(n)-i}(n)E_{i+1}}{c_{N_0(n)-i-1}(n)}
=L_n^2(\tau_i)-2 \mu_{N_0(n)-i-1}(n) E_{i+1}.
\end{align*}

Hence, if $L_n(0) = \ell <0$ then for $k=0,1,\dots, N_0(n)-1$,
\begin{align*}
&\{G_n= N_0(n)-k\} \\
&=  \left\{\sum_{i=0}^{k-1}\mu_{N_0(n)-i-1}(n)E_{i+1} < \ell^2/2,
 \sum_{i=0}^{k}\mu_{N_0(n)-i-1}(n)E_{i+1} \geq \ell^2/2\right\}\\
&=  \left\{\sum_{j=N_{0}(n)-k}^{N_0(n)-1}\mu_{j}(n)E_{N_0(n)-j} < \ell^2/2,
 \sum_{j=N_{0}(n)-k-1}^{N_0(n)-1}\mu_{j}(n)E_{N_0(n)-j} \geq \ell^2/2\right\}.
\end{align*}
If we change the variable by taking $m=N_0(n)-k$ then we obtain for $m=1,2,\dots, N_0(n)$,
\begin{align}\label{s29.1}
\{G_n= m\} 
=  \left\{\sum_{j=m}^{N_0(n)-1}\mu_{j}(n)E_{N_0(n)-j} < \ell^2/2,
 \sum_{j=m-1}^{N_0(n)-1}\mu_{j}(n)E_{N_0(n)-j} \geq \ell^2/2\right\}.
\end{align}
Set 
\begin{align*}
Z_m(n)=\sum_{j=m}^{N_0(n)-1}\mu_j(n)E_{N_0(n)-j},\qquad m=1,2,\dots, N_0(n).
\end{align*}
It follows from \eqref{s27.1} that the density of $Z_m(n)$ is given by
\begin{align}\label{s29.2}
f_{Z_{m}(n)}(t):=
\left(\prod_{j=m}^{N_0(n)-1}\frac 1{\mu_j(n)} \right)
\sum_{k=m}^{N_0(n)-1} \left(\exp(-t / \mu_k(n)) 
\prod_{\substack{i=m\\i \ne k}}^{N_0(n)-1}
\frac1{1/\mu_i(n) -1/ \mu_k(n)}
\right),
\end{align}
provided that $\mu_j(n)\neq \mu_k(n)$ for $k\neq j$; this is the case because we assumed that $\bold{F2}$ or $\bold{F'}$ holds. 

For $m= N_0(n)$, we do not need \eqref{s29.2}; formula \eqref{s29.1} yields, 
\begin{align*}
\mathbb{P}(G_n=N_0(n)) &= \mathbb{P}\left(\mu_{N_0(n)-1}(n)E_{1}\geq \ell^2/2\right)
= e^{-\mu_{N_0(n)-1}(n)^{-1}\ell^2/2}.
\end{align*}
For $m=1, \cdots, N_0(n)-1$, we use \eqref{s29.1} and \eqref{s29.2} as follows,
\begin{align}\label{o4.3}
&\mathbb{P}  (G_n=k)= \mathbb{P}\left(Z_{m-1}(n)\geq \ell^2/2, Z_{m}(n)<\ell^2/2\right)\\
&= \int_0^{\ell^2/2} \mathbb{P}\left(\mu_{m-1}(n)E_{N_0(n) -m+1}\geq \ell^2/2-u\right)f_{Z_{m}}(u)du \notag \\
&= \int_0^{\ell^2/2} \exp\left(-\mu_{m-1}(n)^{-1}(\ell^2/2-u)\right) \times \notag  \notag \\
&\qquad \times \left(\prod_{j=m}^{N_0(n)-1}\frac1{\mu_j(n)} \right)
\sum_{k=m}^{N_0(n)-1} \left(\exp(-u / \mu_k(n)) 
\prod_{\substack{i=m\\i \ne k}}^{N_0(n)-1}
\frac1{1/\mu_i(n) - 1/\mu_k(n)}
\right)du \notag \\
&= \exp\left(-\mu_{m-1}(n)^{-1}\ell^2/2\right)
\left(\prod_{j=m}^{N_0(n)-1}\frac1{\mu_j(n)} \right) \times \notag \\
&\qquad \times  
\sum_{k=m}^{N_0(n)-1} \left(\int_0^{\ell^2/2}\exp(- (\mu_k(n)^{-1}-\mu_{m-1}(n)^{-1})u) du
\prod_{\substack{i=m\\i \ne k}}^{N_0(n)-1}\frac1{1/\mu_i(n) - 1/\mu_k(n)}
\right) \notag \\
&= \exp\left(-\mu_{m-1}(n)^{-1}\ell^2/2\right)
\left(\prod_{j=m}^{N_0(n)-1}\frac1{\mu_j(n)} \right) \times \notag \\
&\qquad \times  
\sum_{k=m}^{N_0(n)-1} \left(
\frac{1- \exp(- (\mu_k(n)^{-1}-\mu_{m-1}(n)^{-1})\ell^2/2)}
{\mu_k(n)^{-1}-\mu_{m-1}(n)^{-1}}
\prod_{\substack{i=m\\i \ne k}}^{N_0(n)-1}\frac1{1/\mu_i(n) - 1/\mu_k(n)}
\right) \notag \\
&= \left(\prod_{j=m}^{N_0(n)-1}\frac1{\mu_j(n)} \right)\times \notag \\
&\quad \times  
\sum_{k=m}^{N_0(n)-1} \left(
\left(\exp\left(-\mu_{k}(n)^{-1}\ell^2/2\right)- \exp(- \mu_{m-1}(n)^{-1}\ell^2/2)\right)
\prod_{\substack{i=m-1\\i \ne k}}^{N_0(n)-1}\frac1{1/\mu_i(n) - 1/\mu_k(n)}
\right) \notag \\
&= \left(\prod_{j=m}^{N_0(n)-1}\frac1{\mu_j(n)} \right)
\sum_{k=m}^{N_0(n)-1} \left(\exp\left(-\mu_{k}(n)^{-1}\ell^2/2\right)
\prod_{\substack{i=m-1\\i \ne k}}^{N_0(n)-1}\frac1{1/\mu_i(n) - 1/\mu_k(n)}
\right) \notag \\
& \qquad-\left(\prod_{j=m}^{N_0(n)-1}\frac1{\mu_j(n)} \right)
\exp(- \mu_{m-1}(n)^{-1}\ell^2/2)
\sum_{k=m}^{N_0(n)-1} 
\prod_{\substack{i=m-1\\i \ne k}}^{N_0(n)-1}\frac1{1/\mu_i(n) - 1/\mu_k(n)}. \notag 
\end{align}
We now apply \eqref{o1.1} to the last line to obtain
\begin{align*}
\mathbb{P} & (G_n=k)
= \left(\prod_{j=m}^{N_0(n)-1}\frac1{\mu_j(n)} \right) 
\sum_{k=m}^{N_0(n)-1} \left(\exp\left(-\mu_{k}(n)^{-1}\ell^2/2\right)
\prod_{\substack{i=m-1\\i \ne k}}^{N_0(n)-1}\frac1{1/\mu_i(n) - 1/\mu_k(n)}
\right) \notag \\
& \qquad-\left(\prod_{j=m}^{N_0(n)-1}\frac1{\mu_j(n)} \right)
\exp(- \mu_{m-1}(n)^{-1}\ell^2/2)
\prod_{i=m}^{N_0(n)-1}
\frac1{1/\mu_i(n) - 1/\mu_{m-1}(n)}
 \notag \\
&= \left(\prod_{j=m}^{N_0(n)-1}\frac1{\mu_j(n)} \right)
\sum_{k=m-1}^{N_0(n)-1} \left(\exp\left(-\mu_{k}(n)^{-1}\ell^2/2\right)
\prod_{\substack{i=m-1\\i \ne k}}^{N_0(n)-1}\frac1{1/\mu_i(n) - 1/\mu_k(n)}
\right) \notag \\
&= \mu_{m-1}(n)f_{Z_{m-1}}(n)(\ell^2/2).
\end{align*}
This and \eqref{s29.2} yield \eqref{o1.2}. 
Finally, \eqref{s30.3} is true because $p_j(n, \ell) = 0$ for $j\notin \prt \calD_n^-$.
\end{proof}

\begin{proof}[Proof of Proposition \ref{CondExitVel}]

Suppose that $E_j$, $j\geq 0$, are i.i.d. exponential with mean 1. 
Suppose that  $X_n(0)=i\in \prt\calD_n^-$ and  $L_n(0)=\ell\geq 0$, and let $T$ be the time of the first jump, necessarily to $i+1$. We can represent $T$ as follows,
\begin{align}\label{o3.3}
T=\inf\left(t>0: \int_0^t c_i(n)\big(\ell+ v_i(n)s\big)ds\geq E_i\right).
\end{align}
Hence, $T$ is the smallest positive solution to
\begin{align}\label{o3.4}
T\ell + T^2 v_i(n) /2 = E_i/c_i(n).
\end{align}
This yields
\begin{align}\label{TimeToWait}
v_i(n) T &= -\ell +\left(\ell^2 + 2\frac{v_i(n)}{c_i(n)}E_i\right)^{1/2},\\
\label{Pos_initial_velo}
L_n(T)&= \ell+ v_i(n)T=\left(\ell^2 + 2\frac{v_i(n)}{c_i(n)}E_i\right)^{1/2}.
\end{align}

Recall notation from \eqref{s30.21} and \eqref{s30.23}.
By the strong Markov property applied at the stopping time $\mc{T}_n$, the distribution of  $L_n(\mc{U}_n)$ is the same in the following cases: (i)  $X_n(0) =N_0(n)$, $L_n(0) = \ell<0$, and $G_n = k $, and (ii)
$X_n(0) = k$ and $L_n(0) =0$.
Assume (ii).
An application of the strong Markov property at the jump times from $j$ to $j+1$ for $j=k, \dots , N_0(n)$, and \eqref{Pos_initial_velo} show that the distribution of $L_n(\mc{U}_n)$ is the same as that of
\begin{align*}%\label{exitStrength}
\left(2\sum_{j=k}^{N_0(n)}E_{j}\frac{v_{j}(n)}{c_{j}(n)}\right)^{1/2}
=\left(2\sum_{j=k}^{N_0(n)}E_{j}\lambda_{j}(n)\right)^{1/2}.
\end{align*}
This proves \eqref{s29.3}.

To prove \eqref{s29.4}, note that $L_n(\mc{U}_n)^2/2$ can be represented as the sum of independent exponential random variables. Their means are all distinct, i.e., $\lambda_j(n)\neq \lambda_i(n)$ for all $i\neq j$, 
because we assumed that either $\bold{F1}$ or $\bold{F'}$ hold.
Thus, we can use \eqref{s27.1} to conclude that
\begin{align*}
\mathbb{P}&(L_n(\mc{U}_n)\leq r) = \mathbb{P}( L_n(\mc{U}_n)^2/2\leq r^2/2)\\
&= \int_0^{r^2/2}\left(\prod_{j=k}^{N_0(n)}\frac1{\lambda_j(n)} \right)
\sum_{j=k}^{N_0(n)} \left(\exp(-t / \lambda_j(n)) 
\prod_{\substack{i=k\\i \ne j}}^{N_0(n)}
\frac1{1/\lambda_i(n) - 1/\lambda_j(n)}
\right)dt.
\end{align*}
Differentiating the above expression with respect to $r$ yields \eqref{s29.4}.
\end{proof}

\begin{proof}[Proof of Theorem \ref{ThinBoundaries}]

(i) In view of Remark \ref{s30.20} (ii) and explicit formulas \eqref{s30.1}-\eqref{s30.3}, the limit in \eqref{s30.4} must exist, except that we have to show that the limit does not involve division by 0. By Remark \ref{s30.20} (iii), $\mu_i > \mu_{i+1} >0$, so one can take the limit in \eqref{s30.1}-\eqref{s30.3} as $n\to \infty$ and the limiting formulas do not involve division by 0. 

(ii)
It follows from Remark \ref{s30.20} (ii) that, for every $k\leq N$,
\begin{align*}%\label{s24.5}
 \left(2\sum_{j=k}^{N}\lambda_j(n)E_j\right)^{1/2} 
\to 
\left(2\sum_{j=k}^{N}\lambda_j E_j\right)^{1/2} ,
\end{align*}
in distribution. This and part (i) of the theorem easily imply part (ii).

Part (iii) is a special case of part (ii).
\end{proof}

\begin{proof}[Proof of Proposition \ref{CVGE_probs}]

Suppose that $E_j$, $j\geq 1$, are i.i.d. exponential with mean 1. Recall notation from \eqref{o2.1}-\eqref{o2.4}.
Fix some $k\geq 1$ and consider $n$ such that $N_0(n)>k$. Set 
\begin{align*}
&Y_{n,k} = \sum_{j= k}^\infty\mu'_j(n)E_j,
\qquad Y_k = \sum_{j= k}^\infty \mu_j E_j.
\end{align*}
We have assumed $\bold{G1}$ so  $\sum_{j\geq 1} \mu_j < \infty$. This, the fact that $E_j$'s are exponential and Kolmogorov's three-series theorem easily imply that $Y_k$ is well defined and finite, a.s.

Since $\mu_kE_k$ has a density, so does $Y_k =\mu_kE_k+ \sum_{j= k+1}^\infty \mu_j E_j$. Hence, $\P(Y_k= \ell^2/2) =0$ for every $\ell$. Fix some $\ell<0$ and find $\eps>0$ so small that
\begin{align}\label{o2.6}
\mathbb{P}\left( (\ell^2-\epsilon)/2 \leq Y_{k}\leq (\ell^2+\eps)/2\right)\leq \delta .
\end{align}

We apply formula \eqref{s29.1} to see that
\begin{align*}
p_k(n,\ell_n) &=
\mathbb{P}( G_n = k\mid X_n(0) = N_0(n), L_n(0) = \ell_n)\\
&=\mathbb{P}\left(\sum_{j=k-1}^{N_0(n)-1}\mu_j(n)E_j\geq \frac{\ell_n^2}{2}, \sum_{j=k}^{N_0(n)-1}\mu_j(n)E_j<\frac{\ell_n^2}{2}\right)\\
&=\mathbb{P}\left(\sum_{j=k-1}^{N_0(n)-1}\mu_j(n)E_j\geq \frac{\ell_n^2}{2}\right) - \mathbb{P}\left(\sum_{j=k}^{N_0(n)-1}\mu_j(n)E_j\geq \frac{\ell_n^2}{2}\right)\\
&=\mathbb{P}\left(\sum_{j= k-1}^\infty\mu'_j(n)E_j\geq \frac{\ell_n^2}{2}\right) - \mathbb{P}\left(\sum_{j= k}^\infty\mu'_j(n)E_j\geq \frac{\ell_n^2}{2}\right)\\
&= \mathbb{P}\left(Y_{n,k-1}\geq \ell_n^2/2\right)-\mathbb{P}\left(Y_{n,k}\geq \ell_n^2/2\right) .
\end{align*}
It will suffice to prove that, for any fixed $k$, $\mathbb{P}\left(Y_{n,k}\geq \ell_n^2/2\right)$ converges to $\mathbb{P}\left(Y_{k}\geq \ell^2/2\right)$
as $n$ goes to infinity.

Since $\ell_n\to \ell$, we can find $n_1$ so large that $\left\vert\ell_n^2/2- \ell^2/2\right\vert < \eps/2$ for $n\geq n_1$. 
Then, for $n\geq n_1$,
\begin{align*}
\mathbb{P}\left(Y_{n,k}\geq \frac{\ell^2+\eps}{2}\right)
\leq \mathbb{P}\left(Y_{n,k}\geq \frac{\ell_n^2}{2}\right)
\leq \mathbb{P}\left(Y_{n,k}\geq \frac{\ell^2-\eps}{2}\right),
\end{align*}
and, therefore,
\begin{align}\label{o2.5}
&\left\vert \mathbb{P}\left(Y_{n,k}\geq \frac{\ell_n^2}{2}\right) - \mathbb{P}\left(Y_{k}\geq \frac{\ell^2}{2}\right)\right\vert \\
&\leq \max\left(\left\vert \mathbb{P}\left(Y_{n,k}\geq \frac{\ell^2-\eps}{2}\right) - \mathbb{P}\left(Y_{k}\geq \frac{\ell^2}{2}\right)\right\vert ,\ 
\left\vert \mathbb{P}\left(Y_{n,k}\geq \frac{\ell^2+\eps}{2}\right) - \mathbb{P}\left(Y_{k}\geq \frac{\ell^2}{2}\right)\right\vert \right). \notag
\end{align}
We will estimate one of the quantities under ``max'' on the right hand side. The other one can be estimated in a similar way. 

We have, by assumption $\bold{G1}$, 
\begin{align}\label{o2.8}
\E\left|Y_{n,k}-Y_k\right|
&=
\mathbb{E}\left(\left\vert\sum_{j\geq k}\mu'_j(n)E_j- \sum_{j\geq k}\mu_j E_j\right\vert\right)
\leq \mathbb{E}\left(\sum_{j\geq k}\left\vert\mu'_j(n)- \mu_j\right\vert E_j\right)\\
&= \sum_{j\geq k}\left\vert\mu'_j(n)- \mu_j \right\vert 
\to 0, \qquad \text{  when  } n\to\infty.\notag
\end{align}
Hence $Y_{n,k}$ converges in $L^1$, thus in distribution, to $Y_k$ as $n$ goes to infinity. 
Since $Y_k$ has a density,  Portmanteau's theorem implies that there exists $n_2$ such that for all $n\geq n_2$,
\begin{align*}%\label{uPr1b}
\left\vert\mathbb{P}\left( Y_{n,k} \geq \frac{\ell^2-\epsilon}{2}\right)- \mathbb{P}\left( Y_{k} \geq \frac{\ell^2-\epsilon}{2}\right)\right\vert\leq \delta .
\end{align*}
Combined with \eqref{o2.6}, this yields
\begin{align*}%\label{uPr1b}
\left\vert\mathbb{P}\left( Y_{n,k} \geq \frac{\ell^2-\epsilon}{2}\right)- \mathbb{P}\left( Y_{k} \geq \frac{\ell^2}{2}\right)\right\vert\leq 2\delta .
\end{align*}
An analogous estimate holds for the other quantity under ``max'' on the right hand side of \eqref{o2.5} so, for large $n$,
\begin{align*}%\label{o2.5}
&\left\vert \mathbb{P}\left(Y_{n,k}\geq \frac{\ell_n^2}{2}\right) - \mathbb{P}\left(Y_{k}\geq \frac{\ell^2}{2}\right)\right\vert \leq 2 \delta.
\end{align*}
Since $\delta>0$ is arbitrarily small, this completes the proof.
\end{proof}

\begin{lemma}\label{o16.1}
Suppose that for every $n\geq 1$, random variables $A_{k,n}$, $k\geq1$, and  $Z_n$ are  defined on the same probability space, and for each $n$, $Z_n$ is independent of $A_{k,n}$, $k\geq1$. 
Suppose that  $A_k$, $k\geq 1$,  and $Z$ are also defined on the same probability space and $Z$ is independent of $A_k$, $k\geq 1$.
Assume that $Z$ and $Z_n$ take only strictly positive integer values, for each $n$, a.s. 
Suppose that $A_{k,n} \to A_k$ and $Z_n \to Z$, in distribution, as $n\to \infty$, for every $k$.
 Let
\begin{align*}
S_n = \sum_{k=1}^\infty A_{k,n} \bone_{\{Z_n=k\}}, \qquad 
S = \sum_{k=1}^\infty A_{k} \bone_{\{Z=k\}}.
\end{align*}
Then $S_n$ converges to $S$ in distribution, as $n\to\infty$.
\end{lemma}

\begin{proof}
The proof is routine so we only sketch it. 
For any $\eps >0$, there is $k_0$ such that $\P(Z\geq k_0) < \eps$. Hence, there is $n_0$ such that for $n\geq n_0$, $\P(Z_n\geq k_0) < 2\eps$.
This implies that it will suffice to show that
\begin{align*}
 \sum_{k=1}^j A_{k,n} \bone_{\{Z_n=k\}} \to 
S = \sum_{k=1}^j A_{k} \bone_{\{Z=k\}},
\end{align*}
in distribution, for every fixed $j\geq 1$.

For any random variable $X$, let $\phi_{X}(t)$ denote its characteristic function. 
We need to show that
\begin{align*}
 \E\left(\sum_{k=1}^j \phi_{A_{k,n}}(t) \bone_{\{Z_n=k\}} \right)\to 
\E\left(\sum_{k=1}^j \phi_{A_{k}}(t) \bone_{\{Z=k\}}\right),
\end{align*}
for every real $t$, as $n\to \infty$. This follows from (i) pointwise convergence $\phi_{A_{k,n}}(t) \to \phi_{A_{k}}(t)$, (ii) the Skorokhod representation theorem which lets us assume that $Z_n \to Z$, a.s., and (iii) dominated convergence theorem. 
\end{proof}

\begin{proof}[Proof of Theorem \ref{Infinite_layers}]

Recall notation and definitions from \eqref{o2.1}-\eqref{o2.4} and \eqref{s30.21}-\eqref{s24.1}.
Suppose that $\ell_n < 0$ for $n\geq 1$, $\lim_{n\rightarrow \infty}\ell_n =\ell < 0$, and $\lim_{n\to \infty } N_0(n) =\infty$. 
Assume that  $X_n(0) =N_0(n)$ and $L_n(0) = \ell_n$ for all $n$.
By Corollary \ref{o2.7},
the distribution of $L_n(\mc{U}_n)$ is the same  as that of
\begin{align*}
\sum_{j=0}^\infty \left(2\sum_{i=j}^\infty\lambda'_i(n) E_i\right)^{1/2}
\bone_{Z_n(\ell_n)=j},
\end{align*}
where $E_1, E_2, \dots$ are i.i.d. exponential random variables with mean 1; $Z_n(\ell_n)\geq 0$ is an integer valued random variable, independent of $E_j$'s and such that
 $\mathbb{P}(Z_n(\ell_n)=j)=p_j(n,\ell_n)$, $ j\geq0$.
In other words, 
if $\mc{L}_{j,n}$ denotes the distribution of $\left(2\sum_{i=j}^\infty\lambda'_i(n) E_i\right)^{1/2}$ and $\nu_n$ denotes the distribution of $Z_n(\ell_n)$ then 
$\mc{V}_n(\ell_n)$ is a mixture of distributions $\mc{L}_{j,n}$ with the mixing measure $\nu_n$ for the index $j$. 

Let $\mc{L}_{j}$ denote the distribution of $\left(2\sum_{i=j}^\infty\lambda_i E_i\right)^{1/2}$ and let $\nu$ denote the distribution of $Z(\ell)$. 
The argument given in \eqref{o2.8} shows that $\mc{L}_{j,n}\to \mc{L}_{j}$ for every $j$, except that we have to replace $\mu$'s with $\lambda$'s, and use assumption $\bold{G2}$.
Distributions $\nu_n$ converge to $\nu$ by Proposition \ref{CVGE_probs}.
We use Lemma \ref{o16.1} to conclude that $L_n(\mc{U}_n)$ converge in distribution to the mixture of distributions $\mc{L}_{j}$ with the mixing measure $\nu$ for the index $j$.
This is equivalent to the statement of the theorem.
\end{proof}

\begin{lemma}\label{AgainstFlow} 
Consider the noisy model and
recall that in the noisy case we assume that $N_0(n) =1$. 
Let $T_1$ denote the time of the first jump of $X_n$. We have
\begin{align*}
\mathbb{P}\left( X_n(T_1) =0 \mid X_n(0)= 1, \; L_n(0)= \ell > 0\right)
&=\frac{b_1(n)}{c_1(n)+b_1(n)}.
\end{align*}
\end{lemma}

\begin{proof}
 
 Recall from \eqref{o3.1}-\eqref{o3.2} that, for $\ell>0$, the jump rates are $a_{1,0}(\ell)= b_1(n)|\ell|$ and $a_{1,2}(\ell)= c_1(n)|\ell|$. 
Suppose that $E_1$ and $E_1'$ are  independent exponential random variables with parameter $1$.
An argument similar to that in \eqref{o3.3}-\eqref{TimeToWait} yields the following representation of the probability in question,
\begin{align*}
&\mathbb{P}\left( X_n(T_1) =0 \mid X_n(0)= 1, \; L_n(0)= \ell > 0\right)\\
 &=\mathbb{P}\left(-\frac{\ell}{v_1(n)}+\frac{1}{v_1(n)}
\left(\ell^2 + 2\frac{v_1(n)}{b_1(n)}E_1\right)^{1/2}
<-\frac{\ell}{v_1(n)}+\frac{1}{v_1(n)}
\left(\ell^2 + 2\frac{v_1(n)}{c_1(n)}E_1'\right)^{1/2}\right)\\
&= \mathbb{P}\left(\frac{E_1}{b_1(n)}<\frac{E_1'}{c_1(n)}\right)
=\frac{b_1(n)}{c_1(n)+b_1(n)}.
\end{align*}  
\end{proof}

\begin{proof}[Proof of Proposition \ref{Layer_Reverse}]

The proof is similar to that of Proposition \ref{LevelDistrib} so we will only sketch the main steps. The key to our calculation is a formula analogous to \eqref{s29.1}. In the present case, $X_n$ starts from 1 and may jump between 0 and 1 any number of times before $L_n$ changes the sign from negative to positive. Every visit to 0 or 1 is associated with a positive increment of $L_n$. These observations can be implemented as follows.

Assume that $\beta_1(n)= c_0(n)/v_1(n)\neq b_1(n)/v_0(n) = \beta_2(n)$. 
For $k\geq 0$, let 
\begin{align*}
Y_k(n) = \sum_{j=1}^k \frac{v_1(n)}{c_0(n)} E_j + \sum_{j=1}^k \frac{v_0(n)}{b_1(n)} E_j'
=  \sum_{j=1}^k \beta_1(n)^{-1} E_j + \sum_{j=1}^k \beta_2(n)^{-1} E_j',
\end{align*}
where $(E_j)_j$ and $(E_j')_j$ are i.i.d exponential random variables with parameter $1$. 
It follows from Remark \ref{o4.1} (iii) with $r=2$ and  $k_1=k_2=k$ that for $k\geq 1$, the density function  of $Y_k(n)$  is
\begin{align}\label{o4.2}
f_{k,n}(u)=
 \beta_1(n)^k\beta_2(n)^k\sum_{i=1}^2\sum_{j=1}^k \frac{(-1)^{k-j}}{(j-1)!}u^{j-1}e^{-\beta_i(n)u}
 \binom{2k-j-1}{k-j}\big(\beta_{3-i}(n)-\beta_i(n)\big)^{-(2k-j)}.
\end{align}

The following formula is analogous to \eqref{o4.3},
\begin{align}\label{Prob_rev_@1}
p_1(n,\ell) &=\sum_{k\geq 0}\mathbb{P}\left(Y_k(n) + \beta_1(n)^{-1}E_{k+1}\geq \ell^2/2\; ,\; Y_k(n)<\ell^2/2\right).
\end{align} 
A single term in the sum has the following representation,
\begin{align}\label{at_kth_step}
&\mathbb{P}\left(Y_k(n) + \beta_1(n)^{-1}E_{k+1}\geq \ell^2/2\; ,\; Y_k(n)<\ell^2/2\right)\\
&= \int_0^{\ell^2/2} \mathbb{P}\left(\beta_1(n)^{-1} E_{k+1}\geq \ell^2/2-u\right)f_{k,n}(u)du\notag\\
&=\int_0^{\ell^2/2} \exp\left(-\beta_1(n)(\ell^2/2-u)\right) f_{k,n}(u)du\notag\\
&=\beta_1(n)^k\beta_2(n)^k \exp(-\beta_1(n)\ell^2/2)\int_0^{\ell^2/2}\Bigg[\sum_{i=1}^2\sum_{j=1}^k \frac{(-1)^{k-j}}{(j-1)!}u^{j-1}
\exp(-(\beta_i(n)-\beta_1(n))u)\notag\\
&\qquad\qquad\qquad\qquad\qquad\qquad\qquad\times \binom{2k-j-1}{k-j}\big(\beta_{3-i}(n)-\beta_i(n)\big)^{-(2k-j)}\Bigg]du.\notag
\end{align}
This and \eqref{Prob_rev_@1} prove part (i) of proposition.

If $\beta_1(n)= c_0(n)/v_1(n)= b_1(n)/v_0(n) = \beta_2(n)$ then $Y_k(n)$ is 
the sum of $2k$ i.i.d. exponential random variables with parameter $\beta_1(n)$ and, therefore, it has the following Gamma density,
\begin{align*}
f_{k,n}(u) =\frac{\beta_1(n)^{2k} u^{2k-1}}{(2k-1)!}\exp(-\beta_1(n) u).
\end{align*}
It follows that
\begin{align}\label{o4.4}
&\mathbb{P}\left(Y_k(n) + \beta_1(n)^{-1}E_{k+1}\geq \ell^2/2\; ,\; Y_k(n)<\ell^2/2\right)\\
&= \int_0^{\ell^2/2} \mathbb{P}\left(\beta_1(n)^{-1} E_{k+1}\geq \ell^2/2-u\right)f_{k,n}(u)du\notag\\
&=\int_0^{\ell^2/2} \exp\left(-\beta_1(n)(\ell^2/2-u)\right) 
\frac{\beta_1(n)^{2k} u^{2k-1}}{(2k-1)!}\exp(-\beta_1(n) u) du\notag\\
&= \frac{\left(\beta_1(n)\ell^2/2\right)^{2k} }{(2k)!}
\exp(-\beta_1(n) \ell^2/2).\notag
\end{align}
The second part of the proposition follows from this formula and \eqref{Prob_rev_@1}.
\end{proof}

\begin{proof}[Proof of Corollary \ref{Limit_Layer_reverse}]

We have to prove that we can pass to the limit in formulas given in Proposition \ref{Layer_Reverse}. We will refer to the proof of that proposition below.
It follows easily from assumptions $\bold{F3}$ and $\bold{K}$ and explicit formulas in \eqref{at_kth_step} and \eqref{o4.4} that for every $\ell<0$ and $k\geq 0$,
\begin{align}\label{o4.6}
\lim_{n\rightarrow \infty}
\mathbb{P}\left(Y_k(n) + \beta_1(n)^{-1}E_{k+1}\geq \ell^2/2\; ,\; Y_k(n)<\ell^2/2\right)
\end{align}
exists.
For $k\geq 1$, we have
\begin{align}\label{o4.5}
\mathbb{P}&\left(Y_k(n) + \beta_1(n)^{-1}E_{k+1}\geq \ell^2/2\; ,\; Y_k(n)<\ell^2/2\right)
\leq \mathbb{P}\left( Y_k(n)<\ell_n^2/2\right)\\
&\leq \mathbb{P}\left(\beta_1(n)^{-1}E_j<\ell_n^2/2 \;\forall\; j=1,\cdots k\right)
= \left(1-e^{-\beta_1(n)\ell_n^2/2}\right)^k.\notag
\end{align}
The assumptions made in the corollary and $\bold{F3}$ imply that for some $\gamma_1< \infty$ and $n_1$, we have $\beta_1(n)\ell_n^2/2 < \gamma_1$ for all $n\geq n_1$. Hence, there exists $q\in (0,1)$ such that $\Big(1-e^{-\beta_1(n)\ell_n^2/2}\Big)\leq q$ for $n\geq n_1$. This and \eqref{o4.5} imply that the series in \eqref{Prob_rev_@1} is dominated by a geometric series. Thus, in view of \eqref{o4.6}, the limit stated in the corollary exists.
\end{proof}

\begin{proof}[Proof of Proposition \ref{Exit_Velo}]

The proof of the proposition is analogous to that of Corollary \ref{o2.7}. The evolution of the process is split into two parts by the stopping time $\mc{T}_n$. The  pre-$\mc{T}_n$ evolution is captured by Proposition \ref{Layer_Reverse}. The amount accumulated by $L_n$ between times $\mc{T}_n$ and $\mc{U}_n$ can be represented by a formula analogous to \eqref{s24.5} in the noiseless case. In the present case, $X_n$ can jump between 0 and 1 even if $L_n$ is positive so, to account for these jumps, we have to have two sequences of exponential random variables, representing repeated visits at 0 and 1. Once $L_n$ becomes positive, the number of jumps between 0 and 1 is geometric with the parameter determined in Lemma \ref{AgainstFlow}.
 We leave the details of the proof to the reader.
\end{proof}

\begin{proof}[Proof of Theorem \ref{Cvge_Veloc}]

The limits in \eqref{o5.1}-\eqref{o5.2} exist because we assumed
 $\bold{F3}$ and $\bold{K}$.

To prove the second claim of the theorem, note that 
the distributions of random variables in \eqref{s24.3} and \eqref{SnL} are mixtures, with mixing measures being the  distributions of random variables $Z$ and $J(n)$. Due to convergence of the parameters stated in \eqref{o5.1}-\eqref{o5.2}, the distributions of the individual components
$2\frac{v_1(n)}{c_1(n)}E', 2\frac{v_1(n)}{b_1(n)}E_j'$ and $ 2\frac{v_0(n)}{c_0(n)}E_k''$
 in the mixtures converge to the limits $\gamma_1 E',\gamma_0 E_j'$ and $ \gamma_2 E_j''$, which are the terms of the sum in \eqref{o5.7}.
The distributions of the mixing random variables, $Z$ and $J(n)$, converge due to Corollary \ref{Limit_Layer_reverse} and \eqref{o5.2}.
This proves that the mixtures converge, and this is just a different way of expressing the theorem.
\end{proof}

\begin{proof}[Proof of Proposition \ref{timeOnBoundary}]
(i)
First we will consider the noiseless model.

Fix some $\ell <0$ and assume that $X_n(0) = N_0(n)$ and $L_n(0) = \ell$. It is routine to modify our argument for the case $\ell >0$ and  $X_n(0) \ne N_0(n)$. Under these assumptions, $s_0(n) = 0$. We will estimate $t_1(n) - s_0(n) = t_1(n)$. Recall notation from \eqref{s30.21} and \eqref{s30.23} and note that $[s_0(n) , t_1(n)] = [s_0(n), \mc{T}_n) \cup [\mc{T}_n, \mc{U}_n]$.
 
The process $X_n$ will jump toward 0 until it reaches a random point $G_n$ (see \eqref{s30.22} for the definition) and then $X_n$ will jump away from 0 until it exits $\prt \calD_n^-$ at time $\mc{U}_n$. 

Fix some site $i\in \prt \calD_n^- \setminus\{0\}$ and
suppose that $X_n$ arrives at  $i$ at a time $u_-$, and $L_n(u_-) =\ell_i<0$. Let 
\begin{align*}
\Delta u_i^- = \inf\{t\geq 0: X_n(u_- + t)\ne i \text{  or  }
u_- +t = \mc{T}_n\}.
\end{align*}
The following representation of $\Delta u_i^-$ is the same as in the proof of Proposition \ref{LevelDistrib}, except that we are using different notation.
Consider an exponential random variable $E_i$ with mean 1.  
If $\ell_i^2> 2\frac{v_i(n)}{c_{i-1}(n)}E_i $ then, by \eqref{Time_Spent2},
\begin{align}\label{o7.1}
\Delta u_i^-
&= \inf\{t\geq 0: X_n(u_- + t)\ne i\} = \frac{\vert\ell_i\vert}{v_i(n)} - \frac{1}{v_i(n)}\left(\ell_i^2 -2 \frac{v_i(n)}{c_{i-1}(n)}E_i\right)^{1/2}\\
&<  \left(\frac{2}{v_i(n)c_{i-1}(n)}E_i\right)^{1/2}.\notag
\end{align}
The inequality on the right hand side holds because it is equivalent to   $\ell_i^2> 2\frac{v_i(n)}{c_{i-1}(n)}E_i $, as an elementary calculation shows.
If $\ell_i^2\leq 2\frac{v_i(n)}{c_{i-1}(n)}E_i $ then
\begin{align}\label{o7.2}
\Delta u_i^- = \inf\{t\geq 0: 
u_- +t = \mc{T}_n\}
 = \frac{\vert \ell_i\vert}{v_i(n)}\leq \left(\frac{2}{v_i(n)c_{i-1}(n)}E_i\right)^{1/2}.
\end{align} 
It follows from \eqref{o7.1}-\eqref{o7.2} that
\begin{align}\label{o7.3}
\Delta u_i^- \leq \left(\frac{2}{v_i(n)c_{i-1}(n)}E_i\right)^{1/2}.
\end{align} 

If $i=0$ then $X_n$ can jump to site 1 only after time $\mc{U}_n$ so 
\begin{align}\label{o7.4}
\Delta u_0^- = \inf\{t\geq 0: 
u_- +t = \mc{T}_n\}
 = \frac{\vert \ell_0\vert}{v_0(n)}\leq \frac{\vert \ell\vert}{v_0(n)}.
\end{align} 
The above inequality holds because when $X_n$ is on its way from $N_0(n)$ to 0 and the initial value of $L_n$ is $\ell<0$ then the value of $|L_n|$ can only  decrease.

If $X_n$ does not visit $i$ on its way from $N_0(n)$ to 0 then we let $\Delta u_i^- =0$.
Summing $\Delta u_i^- $ over all $i\in \prt \calD_n^-$, we obtain from \eqref{o7.3} and \eqref{o7.4},
\begin{align}\label{o7.5}
\mc{T}_n = \sum_{i=0}^{N_0(n)} \Delta u_i^-
\leq \frac{\vert \ell\vert}{v_0(n)} + \sum_{i=0}^{N_0(n)}
\left(\frac{2}{v_i(n)c_{i-1}(n)}E_i\right)^{1/2},
\end{align}
where $E_i$ are i.i.d. exponential with mean 1.

Fix some site $i\in \prt \calD_n^- \setminus\{0\}$ and
suppose that $X_n$ arrives at  $i$ at a time $u_+$, but now suppose that $L_n(u_+) =\ell_i>0$. Let 
\begin{align*}
\Delta u_i^+ = \inf\{t\geq 0: X_n(u_- + t)\ne i\}.
\end{align*}
Reasoning as in the previous part of the proof and using  \eqref{Pos_initial_velo}, we obtain
\begin{align}\label{o7.6}
\Delta u_i^+
& = -\frac{\ell_i}{v_i(n)} + \frac{1}{v_i(n)}\left(\ell_i^2 +2 \frac{v_i(n)}{c_{i}(n)}E'_i\right)^{1/2}
\leq  \left(\frac{2}{v_i(n)c_{i}(n)}E'_i\right)^{1/2},
\end{align}
where $E'_i$ is mean one exponential.
The above inequality is elementary.
Since $G_n$ does not have to be 0, we have the following bound
\begin{align}\label{o7.7}
\mc{U}_n -\mc{T}_n\leq \sum_{i=0}^{N_0(n)} \Delta u_i^+
\leq \sum_{i=0}^{N_0(n)}
\left(\frac{2}{v_i(n)c_{i}(n)}E'_i\right)^{1/2},
\end{align}
where $E'_i$ are i.i.d. exponential with mean 1.

Combining \eqref{o7.5} and \eqref{o7.7}, we obtain
\begin{align*}%\label{}
\mc{U}_n 
\leq \frac{\vert \ell\vert}{v_0(n)} + \sum_{i=0}^{N_0(n)}
\left(\frac{2}{v_i(n)c_{i-1}(n)}E_i\right)^{1/2}
+\sum_{i=0}^{N_0(n)}
\left(\frac{2}{v_i(n)c_{i}(n)}E'_i\right)^{1/2},
\end{align*}
where $E_i$ and $E'_i$ are i.i.d. exponential with mean 1.
We now remove the condition $L_n(0)=\ell$ and write 
\begin{align}\label{o7.10}
\mc{U}_n 
\leq \frac{\vert L_n(0)\vert}{v_0(n)} + \sum_{i=0}^{N_0(n)}
\left(\frac{2}{v_i(n)c_{i-1}(n)}E_i\right)^{1/2}
+\sum_{i=0}^{N_0(n)}
\left(\frac{2}{v_i(n)c_{i}(n)}E'_i\right)^{1/2}.
\end{align}
By assumptions $\bold{F3}$ and $\bold{G1}$-$\bold{G2}$, there exists a constant $C>0$ such that
\begin{align*}%\label{}
\mc{U}_n 
\leq \frac{C\vert L_n(0)\vert}{n} + \frac C n \sum_{i=0}^{N_0(n)}
\sqrt{E_i}
+\frac C n \sum_{i=0}^{N_0(n)}
\sqrt{E'_i}.
\end{align*}
The first term on the right hand side converges to 0 in distribution, as $n\to \infty$, because we assumed that 
 the distributions of $ L_n(0)$, $n\geq 1$, are tight. The other two terms converge to 0 in distribution by the law of large numbers, in view of Assumption \ref{o8.1} (ii).
This completes the proof that 
$\lim_{n\to\infty} t_1(n) - s_0(n) =0$ in distribution.

To extend the proof to show that $\lim_{n\to\infty} t_{j+1}(n) - s_j(n) =0$ for $j\geq 1$, it will suffice, by the strong Markov property, to argue that for every fixed $j$, the distributions of $L_n(s_j(n))$, $n\geq 1$, are tight.

By randomizing $\ell$ in Corollary \ref{o2.7}, we see that if $ L_n(0)$, $n\geq 1$, are tight then $L_n(t_1(n))$, $n\geq 1$, are tight. Since $L_n(s_{j+1}(n)) = L_n(t_{j+1}(n))$ for all $j$, we see that $L_n(s_1(n))$, $n\geq 1$, are tight. Then we proceed by induction and use the strong Markov property to conclude that for every $j$,
$L_n(s_j(n))$, $n\geq 1$, are tight and $L_n(t_j(n))$, $n\geq 1$, are tight.

\medskip

(ii) In the noisy case, the argument is very similar to that in the noiseless case so we will only sketch the proof. The new version of the key estimate will be based on
Proposition \ref{Exit_Velo}. We will 
use the representation \eqref{s24.3}-\eqref{SnL}
together with the inequality $\sqrt{x+y}\leq \sqrt{x}+\sqrt{y}$.
We obtain the following
 analogue of \eqref{o7.10},
\begin{align*}%\label{o7.10}
\mc{U}_n 
\leq & \frac{\vert L_n(0)\vert}{v_0(n)} + 
\left(\frac{2}{v_0(n)c_0(n)}\right)^{1/2}\sqrt{E_0} 
+ \left(\frac{2}{v_1(n)c_1(n)}\right)^{1/2}\sqrt{E_1}\\ 
&+ \sum_{j=1}^{J(n)-1}\left(
\left(\frac{2}{v_{1}(n)b_{1}(n)}\right)^{1/2}\sqrt{E_{j}'}
+\left(\frac{2}{v_0(n)c_0(n)}\right)^{1/2}\sqrt{E_{j}''}\right).
\end{align*}
Random variables $E_0$, $E_1$, $(E_k')_k$, $(E_k'')_k$ are i.i.d. exponential with mean 1. Random variable $J(n)$ is geometric  with parameter $c_1(n)/(c_1(n)+b_1(n)) $.  All of these random variables are assumed to be independent.  

Assumptions  $\bold{F1}$-$\bold{F2}$ or $\bold{F'}$, $\bold{F3}$, and $\bold{K}$ imply that there exists $C$ such that
\begin{align*}%\label{o7.10}
\mc{U}_n 
\leq & \frac C n \left (|L_n(0)| + \sqrt{E_0} 
+ \sqrt{E_1} + \sum_{j=1}^{J(n)-1}\left(
\sqrt{E_{j}'}
+\sqrt{E_{j}''}\right)\right),
\end{align*}
and the parameter of $J(n)$ is within $(1/C,C)$.
Just like in part (i) of the proof, we can use tightness arguments and induction to show that $\lim_{n\to\infty} t_{j+1}(n) - s_j(n) =0$ in distribution, for $j\geq 0$.
\end{proof}

\begin{proof}[Proof of Theorem \ref{convProc}] 

(a)
(i) Noiseless case.
Set $N_0(\infty)=\lim_{n\rightarrow\infty}N_0(n)$ and note that, in our models, $N_0(\infty)$ can be finite or infinite. 
In both cases, in view of \eqref{o2.1} and  assumption $\bold{G2}$, we have $\sum_{j=0}^{N_0(\infty)} \lambda_j < \infty$. Hence, 
$\E\left(2\sum_{j = 0}^{N_0(\infty)}\lambda_j E_j\right)^{1/2} < \infty$, and, therefore, $\left(2\sum_{j = 0}^{N_0(\infty)}\lambda_j E_j\right)^{1/2} < \infty$, a.s. This and Theorems \ref{ThinBoundaries} and \ref{Infinite_layers} imply that the distributions $\mc{V}_\infty(\ell)$ and $\mc{V}^+_\infty(\ell)$ are stochastically bounded by a single distribution (not depending on $\ell$) of a finite valued random variable.
It follows from this and \eqref{s26.1}-\eqref{o5.6} that, on some probability space, we can construct  an i.i.d. sequence $A_j$, $j\geq 0$, of strictly positive and finite random variables such that    $L(u_j)\leq A_j$, a.s., for all $j$. Note that $\E(1/A_j) > 0$, possibly $\E(1/A_j) = \infty$.
Consequently, for every $k\geq 2$,
$u_k = u_1+ \sum_{i = 1}^{k-1} 1/|L(u_i)| \geq \sum_{i = 1}^{k-1} 1/A_i$, a.s., and the right hand side approaches infinity, a.s., by the strong law of large numbers. This completes the proof in the noiseless case.

(ii) Noisy case.
In view of \eqref{o5.2}, the distribution of $S$ in \eqref{o5.7} does not depend on $\ell$. If follows from this and \eqref{o5.8} that the distributions $\mc{V}_\infty(\ell)$ and $\mc{V}^+_\infty(\ell)$ are stochastically bounded by a single distribution (not depending on $\ell$) of a finite valued random variable. The rest of the proof is the same as in the noiseless case.

\medskip
(b)
Recall notation from \eqref{o5.10}-\eqref{o5.12}.
We have assumed that 
$(X_n(0)/n, L_n(0))$ converge in distribution to  $(X(0), L(0))$, and $X(0)\in(0,1)$, a.s.
This and Assumption \ref{o8.1} (iii) imply that $L_n(s_0(n)) = L_n(0)$ for large $n$.
By the strong Markov property applied at stopping times $s_0(n)$ and
Theorems \ref{ThinBoundaries} (ii), \ref{Infinite_layers} and \ref{Cvge_Veloc}, $L_n(t_1(n))$ converge in distribution to a random variable, say $R_1$.

We proceed by induction. Note that, due to Assumption \ref{o8.1} (iii) we have $L_n(s_j(n)) = L_n(t_j(n))$ for all $j\geq 1$. Suppose that $L_n(t_j(n))$ converge in distribution to a random variable $R_j$. Then $L_n(s_j(n))$ converge in distribution to  $R_j$. By the strong Markov property applied at  $s_j(n)$ and
Theorems \ref{ThinBoundaries} (ii), \ref{Infinite_layers} and \ref{Cvge_Veloc}, $L_n(t_{j+1}(n))$ converge in distribution to a random variable $R_{j+1}$. To complete our notation, we let $R_0 $ have the same distribution as that of $ L(0)$, the weak limit of $L_n(0)$.
Our argument actually implies a stronger claim, i.e.,  that we can define $R_j$'s on a common probability space so that the joint distribution of $\{R_j, j\geq 0\}$ is the same as that of $R_j$'s defined in 
\eqref{s26.1}-\eqref{s26.3}, assuming that $\ell_0$ in that definition 
is randomized and given the distribution of $L(0)$.

It follows from \eqref{o8.2} and a formula analogous to \eqref{s28.2} that  for every fixed $j\geq 1$, we can represent the time interval $[s_j(n) , t_j(n)]$ as follows,
\begin{align}\label{o9.2}
s_j(n) - t_j(n) = \sum_{i\in \calD_n\setminus(\prt \calD_n^- \cup \prt \calD_n^+)} E_{i,j}/(n \big\vert L_n(t_j(n))\big\vert),
\end{align}
where $E_{i,j}$, $i\in \calD_n\setminus(\prt \calD_n^- \cup \prt \calD_n^+)$, are i.i.d. exponential with mean 1, independent of $L_n(t_j(n))$'s.
For $j=0$, the analogous formula is 
\begin{align*}%\label{o9.1}
s_0(n) - t_0(n) = 
\begin{cases}
\sum_{i \leq X_n(0), x\notin \prt \calD_n^- } E_{i,0}/(nL_n(0)),&
 \text{  if  } L_n(0) < 0,\\
\sum_{i \geq X_n(0), x\notin \prt \calD_n^+ } E_{i,0}/(nL_n(0)),&
 \text{  if  } L_n(0) > 0.
\end{cases}
\end{align*}
It follows from this, the assumption that $(X_n(0)/n,L_n(0)) $ converge weakly to $(X(0),L(0))$, the assumption that $X(0)\in(0,1)$ and $L(0)$  does not take value 0, and the law of large numbers that the following limit exists,
\begin{align}\label{o9.1}
\Delta u_1 := \lim_{n\to\infty}
s_0(n) - t_0(n) = 
\begin{cases}
X(0)/(-L(0)) = X(0)/(-R_0),&
 \text{  if  } L(0) < 0,\\
(1-X(0))/L(0) =(1- X(0))/R_0,&
 \text{  if  } L_n(0) > 0,
\end{cases}
\end{align}
in distribution.
For similar reasons, \eqref{o9.2} yields
\begin{align}\label{o9.3}
\Delta u_j := \lim_{n\to\infty}
s_j(n) - t_j(n) = 1 /| R_j|,
\end{align}
in distribution.
It follows from Proposition \ref{timeOnBoundary}, \eqref{o9.1}-\eqref{o9.3} and the strong Markov property that for $j\geq 0$, the following limits exist,
\begin{align}\label{o9.4}
 \lim_{n\to\infty}
s_j(n) =\lim_{n\to\infty} t_{j+1}(n)
= \sum_{i=1}^{j+1} \Delta u_i = \Delta u_1 + \sum _{i=2}^{j+1}  1 /| R_j| =: u_{j+1},
\end{align}
in distribution. Moreover, by the strong Markov property, we have joint convergence, in the sense that for every $j\geq 1$,
the vectors
\begin{align*}
( t_0(n), s_0(n), t_1(n), s_1(n), \dots,  t_j(n), s_j(n)) 
\end{align*}
converge in distribution to
\begin{align*}
(0, u_1, u_1, u_2, u_2, \dots , u_j,u_j,u_{j+1}).
\end{align*}

In view of part (a), to finish the proof, it will suffice to fix $j\geq 0$ and analyze  trajectories of $(X_n/n,L_n)$ on time intervals $[ t_j(n), s_j(n)]$ and  $[ s_j(n), t_{j+1}(n)]$.

It follows easily from \eqref{o5.10}-\eqref{o5.12}  and Assumption \ref{o8.1} (ii) that 
\begin{align*}
\lim_{n\to \infty} \sup\{ X_n(t)/n: t \in [ s_{2j}(n), t_{2j+1}(n)]\} =0,
\qquad \text{  if  } j\geq 0, L(0)< 0,\\
\lim_{n\to \infty} \sup\{1- X_n(t)/n: t \in [ s_{2j+1}(n), t_{2j+2}(n)]\} =0,
\qquad \text{  if  } j\geq 0, L(0)> 0,\\
\lim_{n\to \infty} \sup\{1- X_n(t)/n: t \in [ s_{2j}(n), t_{2j+1}(n)]\} =0,
\qquad \text{  if  } j\geq 0, L(0)> 0,\\
\lim_{n\to \infty} \sup\{ X_n(t)/n: t \in [ s_{2j+1}(n), t_{2j+2}(n)]\} =0,
\qquad \text{  if  } j\geq 0, L(0)< 0.
\end{align*}

Assumption \ref{o8.1} (iii) implies that $L_n$ does not change its value on the interval $[ t_j(n), s_j(n)]$. Hence, the sequence of jump times of $X_n$ on this interval is a Poisson process, and jumps always take $X_n$ in the same direction.  The same reasoning based on the law of large numbers that is behind \eqref{o9.2} proves that $X_n/n$ converge on $[ t_j(n), s_j(n)]$ to a linear function going either from 0 to 1 or vice versa (depending on the sign of $L(0)$ and, therefore, on the sign of $L_n(t_j(n))$), in the supremum norm, weakly, as $n\to \infty$. This completes the proof of the theorem.
\end{proof}

\begin{proof}[Proof of Theorem \ref{mainTheorem}]

Every family of reflection laws in Definition \ref{def:vel} is the limit of reflection laws for discrete approximations $(X_n,L_n)$, according to Theorem \ref{ThinBoundaries}.  For every family of reflection laws given in Definition \ref{def:vel}  there
exists 
a billiard process $(X(t), L(t))$ with Markovian reflections by Theorem \ref{convProc} (a). 
By  Theorem \ref{convProc} (b) there exists a sequence of processes $(X_n, L_n)$ converging in distribution to $(X, L)$ where 
each $(X_n, L_n)$ satisfies equation \ref{eq1}, by construction. 

By Theorem \ref{Theorem:BurdzyWhite}, every process $(X_n, L_n)$ has $\calU(\calD_n) \times \calN(0, 1)$ as its stationary distribution. Consequently, the limiting 
billiard process $(X, L)$ has $\calU(0, 1) \times \calN(0, 1)$ as its stationary distribution; see the discussion in \cite[Chap. 4]{EthierKurtz}, particularly \cite[Chap. 4, Thm. 9.10]{EthierKurtz}.

%The fact that $\calU(0, 1) \times \calN(0, 1)$ is the unique stationary distribution for $(X, L)$ follows from a standard coupling argument.
In order to prove that $\calU(0, 1) \times \calN(0, 1)$ is the unique stationary distribution for $(X, L)$, first note that $(X,L)$ is Feller (i.e., its semi-group maps continuous bounded functions onto continuous bounded functions). From Theorems \ref{ThinBoundaries}, \ref{Infinite_layers} and  \ref{Cvge_Veloc}, one obtains that for any initial condition $(x,\ell)$ and any non-empty open set $\calK \subset [0,1]\times \mathbb{R}$,
\begin{align*}
\mathbb{P}\Big(\exists t>0 \text{ such that } (X_t,L_t)\in\calK\Big)>0.
\end{align*} 
Therefore the support of any invariant probability measure is $[0,1]\times\mathbb{R}$. Because two distinct ergodic invariant probability measures are singular, it follows there can be only one such probability measure. This implies that there is only one invariant probability measure.
\end{proof}

% \bibliographystyle{plain}
% \bibliography{BilliardMemo_Ref}

\end{document}